\documentclass[reqno]{amsproc}
\usepackage[latin1]{inputenc}
\usepackage{amssymb}
\usepackage{hyperref}
\usepackage{amsmath}
\usepackage{latexsym}
\usepackage{cite}
\usepackage{amssymb}
\usepackage{amsfonts}
\usepackage{amscd}
\usepackage{xcolor}
\usepackage{amstext,amsmath,amssymb,amsfonts}
\newtheorem{theorem}{Theorem}

\newtheorem{corollary}[theorem]{Corollary}

\newtheorem{definition}[theorem]{Definition}
\newtheorem{example}[theorem]{Example}

\newtheorem{lemma}[theorem]{Lemma}

\newtheorem{proposition}[theorem]{Proposition}
\newtheorem{remark}[theorem]{Remark}

\newcommand{\C}{\mathbb{C}}

\newcommand{\A}{\mathcal{A}}

\newcommand{\beq}{\begin{eqnarray}}
\newcommand{\eeq}{\end{eqnarray}}
\newcommand{\beqs}{\begin{eqnarray*}}
\newcommand{\eeqs}{\end{eqnarray*}}
\newcommand{\bpro}{\begin{pro}}
\newcommand{\epro}{\end{pro}}
\newcommand{\blem}{\begin{lem}}
\newcommand{\elem}{\end{lem}}
\newcommand{\bdfn}{\begin{dfn}}
\newcommand{\edfn}{\end{dfn}}
\newcommand{\bcor}{\begin{cor}}
\newcommand{\ecor}{\end{cor}}
\newcommand{\bthm}{\begin{thm}}
\newcommand{\ethm}{\end{thm}}
\newcommand{\bex}{\begin{ex}}
\newcommand{\eex}{\end{ex}}
\newcommand{\brmk}{\begin{rmk}}
\newcommand{\ermk}{\end{rmk}}
\newcommand{\bpr}{\begin{pr}}
\newcommand{\epr}{\end{pr}}
\newcommand{\benum}{\begin{enumerate}}
\newcommand{\eenum}{\end{enumerate}}
\newcommand{\bitem}{\begin{itemize}}
\newcommand{\eitem}{\end{itemize}}

\newcommand{\cqfd}{\hfill{\square}}
\chardef\bslash=`\\
\numberwithin{equation}{section}
\numberwithin{table}{section}
\numberwithin{theorem}{section}
\DeclareMathOperator{\id}{id}

\title[Quadratic and sympletic antiassociative algebras ]{Double constructions of quadratic and sympletic antiassociative algebras \footnote{Preprint: ICMPA-MPA/2015/08 } }

\author{Gb\^ev\`ewou Damien  Houndedji$^\dagger$}
\address[$\dagger$]{University of Abomey-Calavi,
International Chair in Mathematical Physics and Applications,
ICMPA-UNESCO Chair, 072 BP 50, Cotonou, Rep. of Benin}
\email{ houndedjid@gmail.com}

\author{Cyrille Essossolim Haliya$^\ast$}
\address[$\ast$]{University of ...,
}
\email{cyzille19@gmail.com}
\begin{document}
\maketitle

\today

\bigskip
\begin{abstract}
This work addresses some relevant characteristics and properties of $q$-generalized associative algebras and $q$-generalized dendriform algebras such as bimodules, matched pairs. We construct for the special case of $q=-1$ an antiassociative algebra with a decomposition into the direct sum of the underlying vector spaces of another antiassociative algebra and its dual such that both of them are subalgebras and the natural symmetric bilinear form is invariant or the natural antisymmetric bilinear form is sympletic. The former is called a double construction of quadratic antiassociative algebra and the later is a double construction of sympletic antiassociative algebra which is interpreted in terms of antidendrifom algebras. We classify the 2-dimensional antiassociative algebras and thoroughly give some double constructions of quadratic and sympletic antiassociative algebras
%Besides, a special emphasis is given to  double constructions of quadratic antiassociative algebras $(q=-1)$ and sympletic antiassociative algebras in terms of antidendriform algebras. Finally, after antiassociative's algebras classification in dimension two, we thoroughly give some double constructions of quadratic and sympletic antiassociative algebraic structures.
  \\
{
{\bf Keywords.}
Antiassociative algebra, antidendriform, sympletic form, $ \mathcal{O} $-operator}\\
{\bf  MSC2010.}  16T25, 05C25, 16S99, 16Z05.
\end{abstract}

\section{Introduction}

Antiassociative algebras first arise in the litterature  specially in \cite{[OK]}. In their paper Okubo and Kamiya gave the essential properties of antiassociative algebras and introduced Jordan-Lie (super) algebras which is  intimately  related  to both Jordan-super and antiassociative  algebras. In 2014, M. Markl and E. Remm in \cite{[MAR]} formulated and proved results concerning Koszulness of operads for n-ary algebras where they focused on the particular case of antiassociative operation, i.e. an operation $(a,b)\to ab$ satisfying $(ab)c+a(bc)=0$ for each $a,b$ and $c$. They showed that the  corresponding operad is not Koszul, hence while the standard cohomology has no sensible meaning, the deformation cohomology coincides with the triple cohomology \cite{[FOX],[FOX1]} and governs deformations of antiassociative algebras.

 Recently, P. Zummanovich  shows that there's a strong link in between antiassociative algebras and Mock-Lie algebras (see \cite{[Zumano]} and \cite{[CAMA]} for more details). 
 
  From the above mentionned litterature, succeding in providing an antiassociative algebraic structure particularly in low dimension may give a good framework for a better description to the above mathematical constructions. 
  %We will provide in this work a survey of an antiassociative algebraic structures antiassociative algebras are used in the construction of Mock-Lie algebras and in deformations of antiassociative algebras. Therefore, 

 In \cite{[C.Bai5]} C. Bai discusses double constructions of Frobenius algebras and antisymetric infinitesimlal bialgebras. Recall that a (symmetric) Frobenius algebra is an associative algebra  with a non-degenerate (symmetric) invariant bilinear form.

  Prior in this work, we will define the $q$-generalized associative algebras and the $q$-generalized dendriform algebras and derive their related algebraic structures such as bimodules and matched pairs. Then, we will discuss the double constructions of quadratic antiassociative algebras the way Bai treated double constructions of frobenius algebras. Furthermore, we will discuss the link between $ \mathcal{O} $-operator and antidendriform algebras. %The double construction of antisymetric infinitesimal antiassociative bialgebras.
  % Furthermore, a classification of antiassociative algebras in low dimension is given with a special emphasis on, their corresponding  antisymetric infinitesimal antiassociative bialgebra structures and, the double construction of quadratic antiassociative algebra are thoroughly discussed.
  
  An $ \mathcal{O} $-operator associated to a bimodule $(l,r,V)$ of an associative algebra A is a linear map $T:V\to  \mathcal{A} $ satisfying
  \beqs
  T(u)\cdot T(v)=T(l(T(u))v+r(T(v)u)), \quad\quad u,v\in V.
  \eeqs
  The notion of $ \mathcal{O} $-operator was introduced in  \cite{[BGN1]} (such a
structure appeared independently in \cite{[U]} under the name of generalized Rota-Baxter operator)
which is an analogue of the $ \mathcal{O} $-operator defined by Kupershmidt as a natural generalization of the operator form of the classical Yang-Baxter equation (\cite{[KUP]} and a further study in \cite{[Bai1]}). Conversely, the antisymmetric part of an $ \mathcal{O} $-operator satisfies the associative Yang-Baxter equation in a larger associative algebra.

	We will show that from an $ \mathcal{O} $-operator, one can get an antidendriform algebra the same way the dendriform algebra is derived from $ \mathcal{O} $-operator. Dendriform are algebras equiped with an associative product which can be written as a linear combination of nonassociative compositions. They were introduced by Loday with motivation in K-theory and have been studied extensively in several area of mathematics and physics, etc (see \cite{[C.Bai5]} and reference therein for more details). Analogously, we will show that antidendriform algebras are equiped with an antiassociative structure which can be written as a linear combinaition of nonassociative compositions.
	
	Furthermore, we find that there is a compatible antidendriform algebra structure on an antiassociative algebra $ \mathcal{A} $ if and only if there exists an invertible $ \mathcal{O} $-operator of $ \mathcal{A} $, or equivalently, there exists an invertible sympletic form associated to certain suitable bimodule of $ \mathcal{A} $.
	
  The paper is organized as follow. In Sec.2, we define the notion of the $q$-generalized associative algebras and study their related algebraic structures such as bimodules and matched pairs. Sec.3, we give an explicit and systematic study on the double constructions of quadratic antiassociative algebras. In Sect.4, we introduce the notion of $q$-generalized dendriform algebras and give an explicit study on their bimodules and matched pairs. In Sec.5, we introduce the close relations between $ \mathcal{O} $-operators and antidendriform algebras and explicitly study the double construction of sympletic antiassociative algebras. In Sec.6, we investigate the classification of the 2-dimensional antiassociative algebras and, quadratic and sympletic double constructions.
  
% and are used in the definition of Mok-Lie algebras and Jordan-Lie (super) algebras. In fact  Mock-Lie are commutative algebras satisfying the Jacobi identity.

\section{$q$-generalized associative algebras}
\subsection{Preliminaries}
We consider two important non degenerate bilinear forms on an antiassociative algebra given as follows.
\begin{enumerate}
\item[1.] A symmetric bilinear form $B(,)$ on an antiassociative algebra $ \mathcal{A} $ is invariant if 
\beq
B(xy,z)=B(x,yz) \quad \forall x,y\in A.
\eeq
\item[2.] A skew-symmetric bilinear form $\omega(,)$ on an antiassociative algebra $ \mathcal{A} $ is said to be sympletic if $\omega$ satisfies the following identity:
\beq
\omega(xy, z) + \omega(yz, x) + \omega(zx, y)=0,
\eeq 
for all $x, y, z\in\A.$
\end{enumerate}
\begin{definition} 
Let "$\cdot$" be a bilinear product in a vector space $\A.$ Suppose that it satisfies the following law:
\beq\label{ant}
(x\cdot y)\cdot z= -x\cdot(y\cdot z).
\eeq
Then, we call the pair $(\A, \cdot)$ an \textbf{antiassociative algebra}. Combining both associative (q=1) and antiassociative (q=-1) cases, any algebra $\A$ satisfying 
\beqs
(x\cdot y)\cdot z= qx\cdot(y\cdot z), q=1, -1
\eeqs
is called a \textbf{$q$-associative algebra}.
\end{definition}
\begin{lemma}\cite{[OK]}\label{Lem1}
Let $(\A, \cdot)$ be an antiassociative algebra. Then any
product involving four or more elements of $(\A, \cdot)$ is identically zero. For example, we have
$$ (x\cdot y)\cdot(z\cdot w)=0$$
for any $x$, $y$, $z$, $w \in V$.

\end{lemma}
\begin{proof}
For simplicity, set $u=x\cdot y$, $v=z\cdot w$ and $t = y\cdot z$. We then compute
$$(x\cdot y)\cdot(z\cdot w)=u\cdot (z\cdot w)=-(u\cdot z)\cdot w=-[(x\cdot y)\cdot z]\cdot w=[x\cdot(y\cdot z)]\cdot w=(x\cdot t)\cdot w$$
and
$$(x\cdot y)\cdot(z\cdot w)=(x\cdot y)\cdot v=-x\cdot(y\cdot v)=-x\cdot[y\cdot(z\cdot w)]=x\cdot[(y\cdot z)\cdot w]=x\cdot (t\cdot w).$$
Adding both of them gives
$$2(x\cdot y)\cdot(z\cdot w)=(x\cdot t)\cdot w+x\cdot(t\cdot w)=0$$
which implies that
$$2\neq 0\quad and \quad (x\cdot y)\cdot(z\cdot w)=0.$$
It's also clear that 
$$[(x\cdot y)\cdot z]\cdot w=0\quad and \quad x\cdot[(y\cdot z)\cdot w]=0.$$
\end{proof}
\begin{corollary}\cite{[OK]}
Any antiassociative algebra cannot have idempotent element. In particular, it cannot possess the unit element.
\end{corollary}
\begin{proof}
Setting $x=y=z=w=e$ in Lemma\ref{Lem1} we obtain $e^2\cdot e^2=0$. If we consider $e^2=e$ then this leads to $e^2=0$ implies that $e=0$.
\end{proof}
Let's assume the field $K$ is of characteristic $\neq2,3$.
\begin{definition}\cite{[Zumano]}
An algebra $(\A, \diamond)$ over $K$ is called mock Lie if it is commutative:
\beq
x\diamond y=y\diamond x,
\eeq
and satisfies the Jacobi identity:
\beq
(x\diamond y)\diamond z+(z\diamond x)\diamond y+(y\diamond z)\diamond x=0
\eeq
for any $x$,$y$,$z\in \A$.
\end{definition}
\begin{theorem}\cite{[Zumano]}
Given an antiassociative algebra $(\A,\cdot)$, the new algebra $\A^{\dagger}$ with multiplication give by the "anticommutator"
\beqs
a\diamond b=\frac{1}{2}\left(a\cdot b+b\cdot a\right),
\eeqs
is a mock-Lie algebra.
\end{theorem}
\begin{proof}
Let $x,y\in \A^{\dagger}$. Then
\beqs
x\diamond y=\frac{1}{2}(x\cdot y+ y\cdot x)
\eeqs
and
\beqs
y\diamond x=\frac{1}{2}(y\cdot x+ x\cdot y).
\eeqs
Therefore $x\diamond y -y\diamond x=0$ leadto the commutativity of $\diamond$.
On the other hand $\forall x,y,z\in \A^{\dagger}$
\begin{align*}
(x\diamond y)\diamond z&=\left(\frac{1}{2}(x\cdot y+y\cdot x)\right)\diamond z,\\
&=\frac{1}{4}(x\cdot y+y\cdot x)\cdot z+\frac{1}{4}z\cdot(x\cdot y+y\cdot x),\\
&=\frac{1}{4}(x\cdot y)\cdot z+\frac{1}{4}(y\cdot x)\cdot z+\frac{1}{4}z\cdot(x\cdot y)+\frac{1}{4}z\cdot(y\cdot x),\\
&=\frac{1}{4}(x\cdot y)\cdot z+\frac{1}{4}(y\cdot x)\cdot z-(z\cdot x)\cdot y-(z\cdot y)\cdot x\quad antiassociativity\,\, of\,\,\cdot. 
\end{align*}
Similarly,
\begin{align*}
(y\diamond z)\diamond x&=\frac{1}{4}(y\cdot z)\cdot x+\frac{1}{4}(z\cdot y)\cdot x+\frac{1}{4}x\cdot(y\cdot z)+\frac{1}{4}x\cdot(z\cdot y),\\
&=\frac{1}{4}(y\cdot z)\cdot x+\frac{1}{4}(z\cdot y)\cdot x-\frac{1}{4}(x\cdot y)\cdot z+\frac{1}{4}x\cdot(z\cdot y),\quad antiassociativity\,\, of\,\,\cdot
\end{align*}
and
\begin{align*}
(z\diamond x)\diamond y&=\frac{1}{4}(z\cdot x)\cdot y+\frac{1}{4}(x\cdot z)\cdot y+\frac{1}{4}y\cdot(z\cdot x)+\frac{1}{4}y\cdot(x\cdot z),\\
&=\frac{1}{4}(z\cdot x)\cdot y-\frac{1}{4}x\cdot( z\cdot y)-\frac{1}{4}(y\cdot z)\cdot x-\frac{1}{4}(y\cdot x)\cdot z, \quad antiassociativity\,\, of\,\,\cdot.
\end{align*}
Thus, we have 
\begin{align*}
(x\diamond y)\diamond z+(y\diamond z)\diamond x+(z\diamond x)\diamond y&=\frac{1}{4}(x\cdot y)\cdot z+\frac{1}{4}(y\cdot x)\cdot z-\frac{1}{4}(z\cdot x)\cdot y-\frac{1}{4}(z\cdot y)\cdot x\\
&+\frac{1}{4}(y\cdot z)\cdot x+\frac{1}{4}(z\cdot y)\cdot x-\frac{1}{4}(x\cdot y)\cdot z+\frac{1}{4}x\cdot(z\cdot y)\\
&+\frac{1}{4}(z\cdot x)\cdot y-\frac{1}{4}x\cdot( z\cdot y)-\frac{1}{4}(y\cdot z)\cdot x-\frac{1}{4}(y\cdot x)\cdot z,\\
 &=0,
\end{align*}
which proves the Jacobi identity.
\end{proof}
Now, let us give a generalized definition.
\begin{definition} 
Let $(\A, \cdot)$ be an algebra over field $\mathcal{K}.$ $(\A, \cdot)$ is called \textbf{$q$-generalized associative algebra} when it satisfies the following law:
\beq
(x\cdot y)\cdot z= qx\cdot(y\cdot z), q\in\mathcal{K}-\lbrace 0\rbrace.
\eeq
\end{definition}
\begin{example}
Let $(\A, \circ)$ be a Zinbiel algebra, that is for all $x, y, z\in\A,$ we have
\beqs
(x\circ y)\circ z= x\circ(y\circ z) + x\circ(z\circ y).
\eeqs
If "$\circ$" is commutative, we have
\beqs
(x\circ y)\circ z= x\circ(y\circ z) + x\circ(y\circ z)=2x\circ(y\circ z)
\eeqs
Hence, a commutative Zinbiel algebra is a $2$-associative algebra.
 \end{example}
 \subsection{Bimodules and matched pairs of $q$-generalized associative algebras}
\begin{definition}
Let $\mathcal{A}$ be a $q$-generalized associative algebra and let V be a vector space. Let $ l, r ,  \mathcal{A} \rightarrow gl(V) $ be two linear maps. $ V $ (or the pair $(l,r)$, or $ (l, r, V) $) is called a \textit{bimodule} of $  \mathcal{A} $ if 
\begin{center}
$ l(xy)v =ql(x)l(y)v, r(xy)v =q^{-1}r(y)r(x)v, l(x)r(y)v =q^{-1}r(y)l(x)v$
\end{center}
for all $ x, y \in  \mathcal{A}, v \in V $.
\end{definition}

\begin{remark}
Let $\mathcal{A}$ be a $q$-generalized associative algebra and $(l, r, V)$ the bimodule of $\A$ 
\begin{itemize}
\item For the particular case of $q=1$,   $ (l, r, V) $) is  a bimodule of associative algebra $\A$ ie
\beq
l(xy)v =l(x)l(y)v, r(xy)v =r(y)r(x)v, l(x)r(y)v =r(y)l(x)v,\quad \forall x,y\in \A,v\in V,
\eeq
which is well known in the litterature.
\item When $q=-1$,  $ (l, r, V) $) is  a bimodule of antiassociative algebra $\A$ ie 
\beq
l(xy)v =-l(x)l(y)v, r(xy)v =-r(y)r(x)v, l(x)r(y)v =-r(y)l(x)v
\eeq
for all $x,y\in \A,v\in V$.
\item When $q=2$, $ (l, r, V) $) is a bimodule of a commutative Zinbiel algebra of $\A$ ie
\beq
l(xy)v =2l(x)l(y)v, r(xy)v =\frac{1}{2}r(y)r(x)v, l(x)r(y)v =\frac{1}{2}r(y)l(x)v,\quad \forall x,y\in\A,v\in V.
\eeq
\end{itemize}
\end{remark}
\begin{proposition}
$ (l, r, V) $ is a bimodule of a $q$-generalized associative algebra $\mathcal{A}$ if and only if the direct sum $\mathcal{A}\oplus V$ of vectors spaces is turned into a $q$-generalized associative algebra  by defining multiplication in $  \mathcal{A} \oplus V $ by 
\beqs
(x + a) \ast (y + b) = x \cdot y + (l(x)b + r(y)a)
\eeqs
for all $ x, y \in  \mathcal{A}, a, b\in V$.
\end{proposition}
\textbf{Proof:}
We have:
\beqs
 [(x_{1} + v_{1})\ast(x_{2} + v_{2})]\ast(x_{3} + v_{3})
&=&(x_{1}\cdot x_{2})\cdot x_{3} + l(x_{1}\cdot x_{2})v_{3} \cr
&&+  r(x_{3})(l(x_{1})v_{2}) + r(x_{3})(r(x_{2})v_{1}) \cr
&=&qx_{1}\cdot(x_{2}\cdot x_{3}) + ql(x_{1})l(x_{2})v_{3} \cr
&&+ql(x_{1})r(x_{3})v_{2} + qr(x_{2}\cdot x_{3})v_{1}\cr
&=&q(x_{1} + v_{1})\ast[(x_{2} + v_{2})\ast(x_{3} + v_{3})]
\eeqs
for all $x_{1}, x_{2}, x_{3}\in\A,\ v_{1}, v_{2}, v_{3}\in V.$

$ \hfill \square $

We denote such $q$-generalized associative algebra $(\mathcal{A}\oplus V, \ast) $ by $\mathcal{A} \ltimes_{l, r} V$ or simply $  \mathcal{A} \ltimes V $.	
\begin{lemma}
Let $ (l, r, V)$ be a bimodule of a $q$-generalized associative algebra $\mathcal{A}$.
\begin{enumerate}
\item[(i)] Let $ l^{\ast}, r^{\ast} $ : $  \mathcal{A} \rightarrow gl(V^{\ast}) $ be the linear maps given by 
\beq
\langle l^{\ast}(x)u^{\ast}, v \rangle = \langle l(x)v, u^{\ast} \rangle, \langle r^{\ast}(x)u^{\ast}, v \rangle = \langle r(x)v, u^{\ast} \rangle
\eeq
for all $ x \in  \mathcal{A} $, $ u^{\ast} \in V^{\ast}, v \in V $. Then,  $ (q^{-2}r^{\ast}, q^{2}l^{\ast}, V^{\ast}) $ is a bimodule of $  \mathcal{A} $.
\item[(ii)] $ (l, 0, V), (0, r, V), (q^{-2}r^{\ast}, 0, V^{\ast})$ and $ (0, q^{2}l^{\ast}, V^{\ast}) $ are bimodules.
\end{enumerate}
\end{lemma}
\textbf{Proof: }  Let $ (l, r, V)$ be a bimodule of a $q$-generalized associative algebra $\mathcal{A}$. Show that
 $(q^{-2}r^{\ast}, q^{2}l^{\ast}, V^{\ast}) $ is a bimodule of $\mathcal{A}$. Let $ x, y \in  \mathcal{A}, u^{\ast} \in V^{\ast}, v \in V $, we have
\begin{enumerate}
\item[(i)]
\beqs
 \langle q^{-2}r^{\ast}(xy)u^{\ast}, v \rangle =  \langle q^{-2}r(xy)v, u^{\ast} \rangle
                                        = \langle q^{-3}r(y)r(x)v, u^{\ast} \rangle 
                                        = \langle q(q^{-2}r^{\ast})(x)(q^{-2}r^{\ast})(y)u^{\ast}, v \rangle 
\eeqs
 leading to  $ q^{-2}r^{\ast}(xy)u^{\ast}= q(q^{-2}r^{\ast})(x)(q^{-2}r^{\ast})(y)u^{\ast}$;

\item[(ii)]
\beqs
 \langle q^{2}l^{\ast}(xy)u^{\ast}, v\rangle = \langle q^{2}l(xy)v, u^{\ast} \rangle
                                        = \langle q^{3}l(x)l(y)v, u^{\ast} \rangle
                                        = \langle q^{-1}(q^{2}l^{\ast})(y)(q^{2}l^{\ast})(x)u^{\ast}, v \rangle   
\eeqs
 giving  $q^{2}l^{\ast}(xy)u^{\ast} = q^{-1}(q^{2}l^{\ast})(y)(q^{2}l^{\ast})(x)u^{\ast}$;

\item[(iii)]
\beqs
\langle (q^{-2}r^{\ast})(x)(q^{2}l^{\ast})(y)u^{\ast}, v \rangle                                               = \langle l(y)r(x)v, u^{\ast} \rangle
                                                   = \langle q^{-1}r(x)l(y)v, u^{\ast} \rangle
                                                   =\langle q^{-1}(q^{2}l^{\ast})(y)(q^{-2}r^{\ast})(x)u^{\ast},v \rangle 
\eeqs
 providing  that  $(q^{-2}r^{\ast})(x)(q^{2}l^{\ast})(y)u^{\ast}= q^{-1}(q^{2}l^{\ast})(y)(q^{-2}r^{\ast})(x)u^{\ast}$.
Hence,  $ (q^{-2}r^{\ast}, q^{2}l^{\ast}, V^{\ast}) $ is a bimodule of $\mathcal{A}$.
\end{enumerate}
 Similarly, we can show also that $ (l, 0, V), (0, r, V), (q^{-2}r^{\ast}, 0, V^{\ast})$ and $ (0, q^{2}l^{\ast}, V^{\ast}) $
 are  well    bimodules of $  \mathcal{A} $.
\begin{remark}
\begin{itemize}
\item For $q=1$ we obtain a bimodule dual of an associative algebra which is well known in \cite{[C.Bai6]}. 
\item For $q=\pm 1$ the dual bimodule of a bimodule of an antiassociative algebra $\A$ and of a bimodule of an associative algebra $\A$ are equal ie $(r^{\ast}, l^{\ast}, V^{\ast})$.
\end{itemize}
\end{remark}
$ \hfill \square $
\begin{example}
Let $(\mathcal{A}, \cdot)$ be a $q$-generalized associative algebra. Let $ L_{\cdot}(x) $ and $ R_{\cdot}(x) $ denote the left and right multiplication operators, respectively, that is, 
$L_{\cdot}(x)(y)=x\cdot y, R_{\cdot}(x)(y) = y\cdot x$. For any $ x, y \in  \mathcal{A} $. Let $ L_{\cdot} :  \mathcal{A} \rightarrow gl( \mathcal{A}) $ with $ x \mapsto L_{\cdot}(x) $ and 
$ R_{\cdot} :  \mathcal{A} \rightarrow gl( \mathcal{A}) $ with $ x \mapsto R_{\cdot}(x) $ (for every $ x \in  \mathcal{A} $) be two linear maps. 
Then $ (L_{\cdot},0), (0,R_{\cdot}) $ and $ (L_{\cdot}, R_{\cdot})$ are bimodules of $  \mathcal{A} $ too.
\end{example}
\begin{theorem}\label{matched pair theorem}
Let $ ( \mathcal{A}, \cdot) $ and $(\mathcal{B}, \circ) $ be two $q$-generalized associative algebras. Suppose that there 
are linear maps $ l_{ \mathcal{A}}, r_{ \mathcal{A}} :  \mathcal{A} \rightarrow gl( \mathcal{B}) $ and $ l_{ \mathcal{B}}, r_{ \mathcal{B}} :  \mathcal{B} \rightarrow gl( \mathcal{A}) $ such that $ (l_{ \mathcal{A}}, r_{ \mathcal{A}} ) $ is

 a bimodule of $  \mathcal{A} $ and $ (l_{ \mathcal{B}}, r_{ \mathcal{B}}) $ is a bimodule of $  \mathcal{B}, $  
satisfying the following  conditions:
\beq \label{matched pair 1}
l_{ \mathcal{A}}(x)(a \circ b) = q^{-1}l_{ \mathcal{A}}(r_{ \mathcal{B}}(a)x)b + q^{-1}(l_{ \mathcal{A}}(x)a)\circ b, 
\eeq
\beq \label{matched pair 2}
r_{ \mathcal{A}}(x)(a \circ b) = qr_{ \mathcal{A}}(l_{ \mathcal{B}}(b)x)a + q a \circ (r_{ \mathcal{A}}(x)b), 
\eeq
\beq \label{matched pair 3}
l_{ \mathcal{B}}(a)(x \cdot y) = q^{-1}l_{ \mathcal{B}}(r_{ \mathcal{A}}(x)a)y + q^{-1}(l_{ \mathcal{B}}(a)x) \cdot y, 
\eeq
\beq \label{matched pair 4}
r_{ \mathcal{B}}(a)(x \cdot y) = qr_{ \mathcal{B}}(l_{ \mathcal{A}}(y)a)x + q x \cdot (r_{ \mathcal{B}}(a)y), 
\eeq
\beq\label{matched pair 5}
l_{ \mathcal{A}}(l_{ \mathcal{B}}(a)x)b + (r_{ \mathcal{A}}(x)a) \circ b - qr_{ \mathcal{A}}(r_{ \mathcal{B}}(b)x)a - q a \circ (l_{ \mathcal{A}}(x)b) = 0, 
\eeq
\beq \label{matched pair 6}
l_{ \mathcal{B}}(l_{ \mathcal{A}}(x)a)y + (r_{ \mathcal{B}}(a)x) \cdot y - qr_{ \mathcal{B}}(r_{ \mathcal{A}}(y)a)x - q x \cdot (l_{ \mathcal{B}}(a)y) = 0
\eeq
 for any $ x, y \in  \mathcal{A}, a,b \in  \mathcal{B} $. Then, 
 there is a $q$-generalized associative algebra structure on the direct sum $  \mathcal{A} \oplus  \mathcal{B} $ of 
the underlying vector spaces of $  \mathcal{A} $ and $  \mathcal{B} $ given by
\beq \label{product of matched pair}
(x + a) \ast (y + b) = (x \cdot y + l_{ \mathcal{B}}(a)y + r_{ \mathcal{B}}(b)x) + (a \circ b +  l_{ \mathcal{A}}(x)b +  r_{ \mathcal{A}}(y)a)
\eeq
for all $ x, y \in  \mathcal{A}, a,b \in  \mathcal{B} $. We denote this $q$-generalized associative algebra by
 $  \mathcal{A} \bowtie^{l_{ \mathcal{A}}, r_{ \mathcal{A}}}_{l_{ \mathcal{B}}, r_{ \mathcal{B}}}  \mathcal{B} $ 
or simply $  \mathcal{A} \bowtie  \mathcal{B} $. 
\end{theorem}
\textbf{Proof: }  We have:
\beqs
((x + a) \ast (y +b)) \ast (z +c)  &=& [ (x \cdot y + l_{ \mathcal{B}}(a)y + r_{ \mathcal{B}}(b)x) \cr
                                   && + (a \circ b + l_{ \mathcal{A}}(x)b + r_{ \mathcal{A}}(y)a) ] 
                                   \ast (z +c) \cr
                                   &=& (x \cdot y) \cdot z + (l_{ \mathcal{B}}(a)y)\cdot z + (r_{ \mathcal{B}}(b)x)\cdot z \cr 
                                 &&  + l_{ \mathcal{B}}(a \circ b)z + l_{ \mathcal{B}}(l_{ \mathcal{A}}(x)b)z
                                  + l_{ \mathcal{B}}(r_{ \mathcal{A}}(y)a)z \cr
                                 && + r_{ \mathcal{B}}(c)(x \cdot y) + r_{ \mathcal{B}}(c)
                                 (l_{ \mathcal{B}}(a)y) + r_{ \mathcal{B}}(c)(r_{ \mathcal{B}}(b)x) \cr
                                 &&  + a\circ (b \circ c) + (l_{\mathcal{A}}(x)b)\circ c \cr
                                 && + (l_{\mathcal{A}}(x)b)\circ c + (r_{ \mathcal{A}}(y)a)\circ c \cr
                                 &&  + l_{ \mathcal{A}}(x \cdot y)c + l_{ \mathcal{A}}(l_{ \mathcal{B}}(a)y)c
                                  \cr
                                 && + l_{ \mathcal{A}}(r_{ \mathcal{B}}(b)x)c + r_{ \mathcal{A}}(z)(a \circ b)
                                  \cr
                                 && + r_{ \mathcal{A}}(z)(l_{ \mathcal{A}}(x)b) + r_{ \mathcal{A}}(z)(r_{ \mathcal{A}}(y)a)
\eeqs
and
\beqs
q(x + a) \ast [(y + b)\ast (z + c)]  &=&q(x + a) \ast [ (y \cdot z + l_{ \mathcal{B}}(b)z +
 r_{ \mathcal{B}}(c)y) 
\cr
&& + ( b \circ c + l_{ \mathcal{A}}(y)c + r_{ \mathcal{A}}(z)b) ]
                                    \cr &=& qx\cdot (y\cdot z) + q x\cdot (l_{ \mathcal{B}}(b)z) + q x\cdot (r_{ \mathcal{B}}(c)y) \cr
                                   && ql_{ \mathcal{B}}(a)(y \cdot z) + q l_{ \mathcal{B}}(a)
                                   (l_{ \mathcal{B}}(b)z) + q l_{ \mathcal{B}}(a)(r_{ \mathcal{B}}(c)y) \cr
                                   &&  q r_{ \mathcal{B}}(b \circ c)x + q r_{ \mathcal{B}}
                                   (l_{ \mathcal{A}}(y)c)x + q r_{ \mathcal{B}}(r_{ \mathcal{A}}(z)b)x  \cr
                                   && q a\circ b\circ c + q a\circ (l_{ \mathcal{A}}(y)c) +q a\circ (r_{ \mathcal{A}}(z)b)  \cr
                                   && q l_{ \mathcal{A}}(x)(b\circ c) + q l_{ \mathcal{A}}(x)(l_{ \mathcal{A}}
                                   (y)c) q l_{ \mathcal{A}}(x)(r_{ \mathcal{A}}(z)b) \cr
                                  && + q r_{ \mathcal{A}}(y\cdot z)a + q r_{ \mathcal{A}}(l_{ \mathcal{B}}(b)z)a + q r_{ \mathcal{A}}(r_{ \mathcal{B}}(c)y)\circ a.
\eeqs
Then $ ((x + a) \ast (y +b)) \ast (z +c) = q(x + a) \ast ((y + b) \ast (z + c)) $.

$ \hfill \square $
\begin{definition}
Let $( \mathcal{A},\cdot)$ and $( \mathcal{B}, \circ) $ be two $q$-generalized associative algebras. Suppose that there are linear maps
 $ l_{ \mathcal{A}}, r_{ \mathcal{A}} :  \mathcal{A} \rightarrow gl( \mathcal{B}) $ and 
$ l_{ \mathcal{B}}, r_{ \mathcal{B}} :  \mathcal{B} \rightarrow gl( \mathcal{A}) $ such that 
$ (l_{ \mathcal{A}}, r_{ \mathcal{A}}) $ is a bimodule of $  \mathcal{A} $ and $ (l_{ \mathcal{B}}, r_{ \mathcal{B}}) $ 
is a bimodule of $  \mathcal{B} $. If   the equations  (\ref{matched pair 1}) - (\ref{matched pair 6}) are satisfied, then
 $ ( \mathcal{A},  \mathcal{B}, l_{ \mathcal{A}}, r_{ \mathcal{A}}, l_{ \mathcal{B}}, r_{ \mathcal{B}}) $ 
 is called a \textbf{matched pair of $q$-generalized associative algebras}.
\end{definition}

\begin{remark} In the previous definition
\begin{itemize}
\item for $q=1$,  $( \mathcal{A},  \mathcal{B}, l_{ \mathcal{A}}, r_{ \mathcal{A}}, l_{ \mathcal{B}}, r_{ \mathcal{B}})$ is called a matched pair of associative algebras which is well known in \cite{[C.Bai6]};
\item for $q=-1$, $( \mathcal{A},  \mathcal{B}, l_{ \mathcal{A}}, r_{ \mathcal{A}}, l_{ \mathcal{B}}, r_{ \mathcal{B}})$ is called a matched pair of antiassociative algebras;
\item when $q=2$, $( \mathcal{A},  \mathcal{B}, l_{ \mathcal{A}}, r_{ \mathcal{A}}, l_{ \mathcal{B}}, r_{ \mathcal{B}})$ is called a matched pair of a commutative Zinbiel algebras.
\end{itemize}
\end{remark}
\section{Double constructions of quadratic antiassociative algebras}
\begin{definition}
We call $(\mathcal{A}, B)$ a   double
construction of a quadratic antiassociative algebra associated to
$\mathcal{A}_1$ and ${\mathcal A}_1^*$ if it satisfies the conditions
 \begin{enumerate}
 \item[(1)] $ \mathcal{A} = \mathcal{A}_{1}
\oplus \mathcal{A}^{\ast}_{1} $ as the direct sum of vector
spaces; \item[(2)] $\mathcal{A}_1$ and ${\mathcal A}_1^*$ are antiassociative subalgebras of $\mathcal{A}$; \item[(3)] $B$ is the natural non-degerenate invariant symmetric
bilinear form on $ \mathcal{A}_{1} \oplus \mathcal{A}^{\ast}_{1} $
given by
 \beq \label{quadratic form}
 B(x + a^{\ast}, y + b^{\ast}) = \langle x, b^{\ast} \rangle +  \langle a^{\ast}, y \rangle
\eeq
for all  $x, y \in \mathcal{A}_{1}, a^{\ast}, b^{\ast} \in \mathcal{A}^{\ast}_{1}$
where $ \langle  , \rangle $ is the natural pair between the vector space $ \mathcal{A}_{1} $ and its dual space  $ \mathcal{A}^{\ast}_{1} $.
\end{enumerate}
\end{definition} 
Let $ (\mathcal{A}, \cdot) $ be an antiassociative algebra. 
Suppose that there is an antiassociative algebra structure $ " \circ " $ on its dual space $ \mathcal{A}^{\ast} $.
 We construct an antiassociative algebra structure on the direct sum $\mathcal{A}\oplus \mathcal{A}^{\ast}$ of the underlying vector spaces of $ \mathcal{A} $ and $ \mathcal{A}^{\ast} $ such that $ (\mathcal{A}, \cdot) $ and 
($ \mathcal{A}^{\ast}, \circ $) are subalgebras and the symmetric bilinear form on $ \mathcal{A}\oplus \mathcal{A}^{\ast} $ 
given by (\ref{quadratic form}) is invariant. That is, $ (\mathcal{A}\oplus \mathcal{A}^{\ast}, \mathcal{B}) $ is a quadratic antiassociative algebra.
 Such a construction is called a double construction of a 
quadratic antiassociative algebra associated to $ (\mathcal{A}, \cdot) $ and $ (\mathcal{A}^{\ast}, \circ)$ 
and we denote it by $ ( \mathcal{A}\oplus \mathcal{A}^{\ast}, B) $.
\begin{theorem}
Let $ (\mathcal{A}, \cdot) $ be an antiassociative algebra. Suppose that there is an antiassociative algebra structure $ " \circ " $ on its 
dual space $ \mathcal{A}^{\ast} $. Then, there is a double construction of a quadratic antiassociative algebra associated to $ (\mathcal{A}, \cdot) $ 
and $ (\mathcal{A}^{\ast}, \circ) $ if and only if $ (\mathcal{A}, \mathcal{A}^{\ast}, R^{\ast}_{\cdot}, L^{\ast}_{\cdot}, R^{\ast}_{\circ},  L^{\ast}_{\circ})$ 
is a matched pair of antiassociative algebras.
\end{theorem}
\textbf{Proof:}  
Let us consider the four maps
\beqs
&&L^{\ast}_{\cdot}: \mathcal{A} \rightarrow gl(\mathcal{A}^{\ast}), \langle L^{\ast}_{\cdot}(x)u^{\ast}, v \rangle = \langle
L_{\cdot}(x)v, u^{\ast} \rangle = \langle xv, u^{\ast} \rangle,\cr \cr
&&R^{\ast}_{\cdot} : \mathcal{A} \rightarrow gl(\mathcal{A}^{\ast}), \langle R^{\ast}_{\cdot}(x)u^{\ast}, v \rangle = \langle
R_{\cdot}(x)v, u^{\ast} \rangle = \langle vx, u^{\ast} \rangle,\cr \cr
&&R^{\ast}_{\circ} : \mathcal{A}^{\ast} \rightarrow gl(\mathcal{A}), \langle R^{\ast}_{\circ}(x^{\ast})u, 
v^{\ast} \rangle = \langle R_{\circ}(x^{\ast})v^{\ast}, u \rangle = \langle v^{\ast} \circ x^{\ast}, u\rangle, \cr \cr
&&L^{\ast}_{\circ} : \mathcal{A}^{\ast} \rightarrow gl(\mathcal{A}), \langle L^{\ast}_{\circ}(x^{\ast})u,
 v^{\ast} \rangle = \langle L_{\circ}(x^{\ast})v^{\ast}, u \rangle = \langle x^{\ast} \circ v^{\ast},
  u\rangle,
\eeqs 
for all $ x, v, u \in \mathcal{A} $, $ x^{\ast}, v^{\ast}, u^{\ast} \in \mathcal{A}^{\ast} $.
If $ (\mathcal{A}, \mathcal{A}^{\ast}, R^{\ast}_{\cdot}, L^{\ast}_{\cdot}, R^{\ast}_{\circ},  L^{\ast}_{\circ})$ is a matched pair of antiassociative algebras, then  the bilinear form 
$ B(\cdot, \cdot) $ defined by the 
equation (\ref{quadratic form}) is invariant on the antiassociative algebra 
$ \mathcal{A} \bowtie ^{R^{\ast}_{\cdot}, L^{\ast}_{\cdot}} _{R^{\ast}_{\circ},L^{\ast}_{\circ}} \mathcal{A}^{\ast} $ with
its product $ \ast $  given by the equation (\ref{product of matched pair}), that is  $ B[(x + a^{\ast})\ast 
(y + b^{\ast}),(z + c^{\ast})] = B[(x + a^{\ast}), (y + b^{\ast})\ast (z + c^{\ast})] $, where $ B(x + a^{\ast}, y + b^{\ast}) = \langle x, b^{\ast} \rangle + \langle a^{\ast}, y\rangle $, 
for all $ x, y \in \mathcal{A}^{\ast}, a^{\ast}, b^{\ast} \in \mathcal{A}^{\ast}$ and $ (x + a^{\ast})\ast
(y + b^{\ast}) = (x\cdot y + l_{\mathcal{B}}(a)y + r_{\mathcal{B}}(b)x ) + (a\circ b + l_{\mathcal{A}}(x)b + r_{\mathcal{A}}(y)a ) $ with $ l_{\mathcal{A}} = R^{\ast}, r_{\mathcal{A}} = L^{\ast}, l_{\mathcal{B}} = R^{\ast}_{\circ}, r_{\mathcal{B}} = L^{\ast} _{\circ} $. Indeed, we have
\beqs
B[(x + a^{\ast})\ast (y + b^{\ast}), (z + c^{\ast})]  &=& B[(x\cdot y + l_{A^{\ast}}
(a^{\ast})y + r_{A^{\ast}}(b^{\ast})x) + (a^{\ast} \circ b^{\ast} \cr 
&& + l_{\mathcal{A}}(x)b^{\ast} +  r_{\mathcal{A}}(y)a^{\ast}), z + c^{\ast}] 
                                                               \cr &=& \langle (x\cdot y +
                                                                l_{\mathcal{A}^{\ast}}(a^{\ast})y + r_{\mathcal{A}^{\ast}}(b^{\ast})x), c^{\ast} \rangle \cr 
                                                              &&  + \langle (a^{\ast} \circ b^{\ast} + l_{\mathcal{A}}(x)b^{\ast} +  r_{\mathcal{A}}(y)a^{\ast}), z \rangle
                                                              \cr &=& \langle x\cdot y, c^{\ast} \rangle 
                                                              + \langle l_{\mathcal{A}^{\ast}}(a^{\ast})y,
                               c^{\ast} \rangle + \langle r_{\mathcal{A}^{\ast}}(b^{\ast})x) , c^{\ast} \rangle \cr 
                                                              && + \langle a^{\ast} \circ b^{\ast}, z \rangle + \langle l_{\mathcal{A}}(x)b^{\ast}, z \rangle + \langle r_{\mathcal{A}}(y)a^{\ast}), z \rangle
                                                              \cr &=& \langle x\cdot y, c^{\ast} \rangle + 
                                                     \langle c^{\ast} \circ a^{\ast} , y \rangle + \langle b^{\ast} \circ c^{\ast} , x \rangle \cr
                                                             && + \langle a^{\ast} \circ b^{\ast} , z \rangle 
                                                             + \langle z\cdot x , b^{\ast} \rangle +
                                                              \langle y\cdot z , a^{\ast} \rangle
\eeqs 
and
\beqs
B[x + a^{\ast} ,(y + b^{\ast}) \ast (z + c^{\ast})]  &=& B[x + a^{\ast} ,(y\cdot z + l_{\mathcal{A}^{\ast}}(b^{\ast})z + r_{\mathcal{A}^{\ast}}(c^{\ast})y)  \cr
&& + (b^{\ast} \circ c^{\ast} + l_{A}(y)c^{\ast} +  r_{\mathcal{A}}(z)b^{\ast})]
                                                           \cr &=& \langle x ,(b^{\ast} \circ c^{\ast} + l_{\mathcal{A}}(y)c^{\ast} +  r_{\mathcal{A}}(z)b^{\ast}) \rangle \cr
                                                              && + \langle(y\cdot z + l_{\mathcal{A}^{\ast}}(b^{\ast})z + r_{\mathcal{A}^{\ast}}(c^{\ast})y, a^{\ast} \rangle 
                                                           \cr &=& \langle x, b^{\ast}\circ c^{\ast} \rangle + \langle x, l_{\mathcal{A}}(y)c^{\ast} \rangle + \langle x, r_{\mathcal{A}}(z)b^{\ast}) \rangle \cr
                                                              &&   + \langle y\cdot z , a^{\ast} \rangle + \langle l_{\mathcal{A}^{\ast}}(b^{\ast})z, a^{\ast} \rangle + 
                                                              \langle  r_{\mathcal{A}^{\ast}}(c^{\ast})y , a^{\ast} \rangle
                                                            \cr &=& \langle x ,b^{\ast}\circ c^{\ast} \rangle 
                                                            + \langle c^{\ast}, x\cdot y  \rangle + \langle b^{\ast}, z\cdot x \rangle \cr  
                                                             &&  + \langle y\cdot z , a^{\ast} \rangle 
                                                     + \langle a^{\ast}\circ b^{\ast}, z \rangle + \langle c^{\ast}\circ a^{\ast}, y \rangle. 
\eeqs 
Thus,  $ B $ is well invariant. Conversely, set 
\beqs
x \ast a^{\ast} = l_{\mathcal{A}}(x)a^{\ast} + r_{\mathcal{A}^{\ast}}(a^{\ast})x, a^{\ast} \ast x = l_{\mathcal{A}^{\ast}}(a^{\ast})x + r_{\mathcal{A}}(x)a^{\star},
\eeqs
for $ x \in \mathcal{A}, a^{\ast} \in \mathcal{A}^{\ast} $.
 Then, $ (\mathcal{A}, \mathcal{A}^{\ast}, R^{\ast}_{\cdot}, L^{\ast}_{\cdot}, R^{\ast}_{\circ},  L^{\ast}_{\circ}) $ is a matched pair of antiassociative algebras, since the  double construction of the quadratic antiassociative algebra associated to $ (\mathcal{A}, \cdot) $ and $ (\mathcal{A}^{\ast}, \circ) $ produces the equations (\ref{matched pair 1}) - (\ref{matched pair 6}).

$ \hfill \square $

\begin{theorem}
Let $ (\mathcal{A}, \cdot) $ be an antiassociative algebra. Suppose that there is an antiassociative algebra 
 structure $ "\circ" $ on its dual space $ \mathcal{A}^{\ast} $. Then,  $ (\mathcal{A}, \mathcal{A}^{\ast},
  R^{\ast}_{\cdot}, L^{\ast}_{\cdot}, R^{\ast}_{\circ},  L^{\ast}_{\circ}) $ is a matched pair of antiassociative algebras 
  if and only if for any $ x \in \mathcal{A}$ and $ a^{\ast}, b^{\ast} \in \mathcal{A}^{\ast} $,
\beq \label{match. pair dual 1}
R^{\ast}_{\cdot}(x)(a^{\ast} \circ b^{\ast}) = -R^{\ast}_{\cdot}(L^{\ast}_{\circ}(a^{\ast})x)b^{\ast} -
(R^{\ast}_{\cdot}(x)a^{\ast})\circ b^{\ast}, 
\eeq
\beq \label{match. pair dual 2}
 R^{\ast}_{\cdot}(R^{\ast}_{\circ}(a^{\ast})x)b^{\ast} + L^{\ast}_{\cdot}(x)a^{\ast}\circ b^{\ast}=  -L^{\ast}_{\cdot}(L^{\ast}_{\circ}(b^{\ast})x)a^{\ast} - a^{\ast}\circ (R^{\ast}_{\cdot}(x)b^{\ast}).
\eeq
\end{theorem}
\textbf{Proof:}
Obviously, (\ref{match. pair dual 1}) gives (\ref{matched pair 1}) and
 (\ref{match. pair dual 2}) reduces to (\ref{matched pair 5})  when $ l_{\mathcal{A}} = R^{\ast}, r_{\mathcal{A}} = L^{\ast}, l_{\mathcal{B}} = l_{\mathcal{A}^{\ast}} = R^{\ast}_{\circ}, r_{\mathcal{B}} = r_{\mathcal{A}^{\ast}} = L^{\ast}_{\circ} $. Now, show that 
\beqs
(\ref{matched pair 1}) \Longleftrightarrow  (\ref{matched pair 2}) \Longleftrightarrow (\ref{matched pair 3})  \Longleftrightarrow (\ref{matched pair 4})  \cr {\mbox{and}}\;
(\ref{matched pair 5}) \Longleftrightarrow (\ref{matched pair 6}). 
\eeqs
 Suppose  (\ref{matched pair 1})  and (\ref{matched pair 5}) are satisfied and show that one has:
\beqs
L^{\ast}(x)(a^{\ast}\circ b^{\ast}) = -L^{\ast}(R^{\ast}_{\circ}(b^{\ast})x)a^{\ast} - a^{\ast}\circ (L^{\ast}(x)b^{\ast})\cr
R^{\ast}_{\circ}(x\cdot y) = -R^{\ast}_{\circ}(L^{\ast}(x)a^{\ast})y - (R^{\ast}_{\circ}(a)x)\cdot y \cr 
L^{\ast}_{\circ}(a^{\ast})(x\cdot y) = -L^{\ast}_{\circ}(R^{\ast}(y)a^{\ast})x - x\cdot (L^{\ast}_{\circ}(a^{\ast})y) \cr 
R^{\ast}_{\circ}(R^{\ast}(x)a^{\ast})y + (L^{\ast}_{\circ}(a^{\ast})x)\cdot y + L^{\ast}_{\circ}(L(y)a^{\ast})x + x\cdot (R^{\ast}_{\circ}(a)y) = 0.
\eeqs
We have :
\beqs
\langle R^{\ast}(x)a^{\ast}, y \rangle = \langle L^{\ast}(y)a^{\ast}, x \rangle = \langle    y\cdot x, a^{\ast} \rangle ;
\langle R^{\ast}_{\circ}(b^{\ast})x, a^{\ast} \rangle = \langle L^{\ast}_{\circ}(a^{\ast})x, b^{\ast} \rangle = \langle a^{\ast}\circ b^{\ast}, x \rangle
\eeqs
for all $ x, y \in \mathcal{A}, a^{\ast}, b^{\ast} \in \mathcal{A}^{\ast} $.Then
\benum
\item[(i)]
\beqs
\langle R^{\ast}(x)(a^{\ast}\circ b^{\ast}), y \rangle = \langle y\cdot x, a^{\ast}\circ b^{\ast} \rangle 
= \langle L^{\ast}(y)(a^{\ast}\circ b^{\ast}), x \rangle ;\cr 
\langle -R^{\ast}(L^{\ast}_{\circ}(a^{\ast})x)b^{\ast}, y \rangle  
= \langle -L^{\ast}(y)b^{\ast}, L^{\ast}_{\circ}(a^{\ast})x \rangle 
= \langle -a^{\ast}\circ(L^{\ast}(y)b^{\ast}), x\rangle \cr 
\langle -(R^{\ast}(x)a^{\ast})\circ b^{\ast}, y\rangle 
= \langle -R^{\ast}(x)a^{\ast},  R^{\ast}_{\circ}(b^{\ast})y \rangle 
= \langle -L^{\ast}( R^{\ast}_{\circ}(b^{\ast})y)a^{\ast}, x \rangle 
\eeqs
  leading to  (\ref{matched pair 1}) $ \Longleftrightarrow $ (\ref{matched pair 2});
\item[(ii)]
\beqs
\langle L^{\ast}(y)(a^{\ast} \circ b^{\ast}), x\rangle &=& \langle -a^{\ast}\circ (L^{\ast}(y)b^{\ast}), x \rangle + \langle -L^{\ast}(R^{\ast}_{\circ}(b^{\ast})y)\cdot x, a^{\ast} \rangle
                                                \cr    &=& \langle -R^{\ast}_{\circ}(L^{\ast}(y)b^{\ast})x, a^{\ast} \rangle + \langle -( R^{\ast}_{\circ}(b^{\ast})y)\cdot x, a^{\ast} \rangle
                                                \cr    &=& \langle  R^{\ast}_{\circ}(b^{\ast})(y\cdot x), a^{\ast} \rangle                                                   
\eeqs
  giving  (\ref{matched pair 2}) $ \Longleftrightarrow $ (\ref{matched pair 3}); 
\item[(iii)]
\beqs
\langle R^{\ast}(x)(a^{\ast}\circ b^{\ast}), y \rangle &=& \langle -R^{\ast}(L^{\ast}_{\circ}(a^{\ast})x)b^{\ast}, y \rangle + \langle -(R^{\ast}(x)a^{\ast})\circ b^{\ast}, y \rangle
                                                  \cr &=& \langle -y\cdot L^{\ast}_{\circ} (a^{\ast})x, b^{\ast} \rangle + \langle -L^{\ast}_{\circ}(R^{\ast}(x)a^{\ast})y, b^{\ast}\rangle
                                                  \cr &=& \langle L^{\ast}_{\circ}(a^{\ast})(y\cdot x), b^{\ast} \rangle
\eeqs
  providing that   (\ref{matched pair 1}) $  \Longleftrightarrow $ (\ref{matched pair 4}); 
\item[(iv)]
\beqs
&&\langle L^{\ast}(L^{\ast}_{\circ}(b^{\ast})x)a^{\ast}, y \rangle = \langle (L^{\ast}_{\circ}(b^{\ast})x)\cdot y, a^{\ast} \rangle ;
\langle a^{\ast}\circ (R^{\ast}(x)b^{\ast}), y \rangle = \langle R^{\ast}_{\circ}(R^{\ast}(x)b^{\ast})y, a^{\ast} \rangle ;\cr 
&&\langle (L^{\ast}(x)a^{\ast})\circ b^{\ast}, y \rangle = \langle R_{\circ}(b^{\ast})y, L^{\ast}(x)a^{\ast} \rangle 
                                                       = \langle x\cdot (R^{\ast}_{\circ}(b^{\ast})y), a^{\ast} \rangle ; \cr 
&& \langle R^{\ast}(R^{\ast}_{\circ}(a^{\ast})x)b^{\ast}, y \rangle  
     = \langle L^{\ast}(y)b^{\ast}, R^{\ast}_{\circ}(a^{\ast})x \rangle 
     = \langle L^{\ast}_{\circ}(L^{\ast}(y)b^{\ast})x, a^{\ast} \rangle                   
\eeqs
  implying that  (\ref{matched pair 5}) $  \Longleftrightarrow $ (\ref{matched pair 6}). 
\eenum  

$ \hfill \square $

\section{$q$-generalized dendriform algebras}
\subsection{Bimodule and matched pair of q-generalized dendriform algebras}
\begin{definition}
 Let $ \mathcal{A} $ be a vector space over a field $\mathbb{F}$ with two bilinear products denoted by $ \prec $ 
and $ \succ $. Then, $ (\mathcal{A}, \prec, \succ) $ is called a \textbf{q-generalized dendriform algebra} if, for any $ x, y, z \in \mathcal{A} $,
\beqs
&&(x \prec y) \prec z = q x \prec (y \ast z), \cr
&&(x \succ y) \prec z = q x \succ (y \prec z), \cr
&&x \succ (y \succ z ) = q^{-1} ( x \ast y) \succ z,
\eeqs
where 
\beq \label{antiassociative product}
x \ast y = x \prec y + x \succ y
\eeq .
\end{definition}

Let $ (\mathcal{A}, \prec, \succ) $ be a q-generalized dendriform algebra. For any $ x \in \mathcal{A} $, let $ L_{\succ}(x),  R_{\succ}(x) $ and $ L_{\prec}(x),$ $ R_{\prec}(x) $ denote the left and right multiplication operators of $(\mathcal{A}, \prec)$ and $(\mathcal{A}, \succ)$, respectively, that is,
\beq
L_{\succ}(x) y = x \succ y, R_{\succ}(x) y = y \succ x, L_{\prec}(x) y = x \prec y, 
R_{\prec}(x) y = y \prec x,
\eeq
for all $ x, y \in \mathcal{A} $. Moreover, let $ L_{\succ}, R_{\succ}, L_{\prec}, R_{\prec} : \mathcal{A} 
\rightarrow gl(\mathcal{A}) $ be four linear maps with $ x \rightarrow L_{\succ}(x),  x \rightarrow 
R_{\succ}(x),  x \rightarrow L_{\prec}(x), $ and $  x \rightarrow R_{\prec}(x)  $, respectively. 
\begin{proposition}  \label{pro2}
Let $(\mathcal{A}, \prec, \succ)$ be a q-generalized dendriform algebra. 
Then, $ (\mathcal{A}, \ast) $ is a q-generalized associative algebra. Moreover, $ (L_{\succ}, R_{\prec}) $ is a bimodule of the 
associated q-generalized associative algebra $ (\mathcal{A}, \ast) $.
\end{proposition}
\textbf{Proof: }
   We have :
\beqs
(x\ast y)\ast z &=& q x\ast (y\ast z) \Longleftrightarrow \cr
(x\ast y)\prec z + (x\ast y)\succ z &=& q x\prec(y\ast z) + q x\succ(y\ast z)\Longleftrightarrow \cr
(x\prec y)\prec z + (x\succ y)\prec z + (x\ast y)\succ z &=& q x\prec(y\ast z) + q x\succ(y\prec z) + q x\succ(y\succ z) \Longleftrightarrow\cr
&&\begin{cases}
~(x \prec y) \prec z = q x \prec (y \ast z),\\
~(x \succ y) \prec z = q x \succ (y \prec z),\\
~x \succ (y \succ z ) = q^{-1} ( x \ast y) \succ z.
\end{cases}
\eeqs
Then, $(\mathcal{A}, \ast)$ is well a q-generalized associative algebra.

Now, show that $ (L_{\succ}, R_{\prec}, \mathcal{A}) $ is a bimodule of q-generalized associative algebra $ (\mathcal{A}, \ast) $:
\beqs
L_{\succ}(x \ast y)v &=& (x \ast y) \succ v= q x\succ(y \succ v) \cr
                     &=& q L_{\succ}(x) L_{\succ}(y)v.                      
\eeqs
 Hence,  $ L_{\succ}(x \ast y)v = q L_{\succ}(x) L_{\succ}(y)v $.  Moreover, 
\beqs
R_{\prec}(x \ast y)v = v\prec (x \ast y) = q^{-1} (v \prec x)\prec y = q^{-1} R_{\prec}(y) 
R_{\prec}(x)v.
\eeqs 
 Then,  $ R_{\prec}(x \ast y)v = q^{-1} R_{\prec}(y) R_{\prec}(x)v. $  It follows that 
\beqs
L_{\succ}(x)R_{\prec}(y)v = x \succ(v \prec y) = q^{-1} (x \succ v)\prec y 
= q^{-1} R_{\prec}(y)L_{\succ}(x)v  .
\eeqs
 Therefore,  $ L_{\succ}(x)R_{\prec}(y)v = q^{-1} R_{\prec}(y)L_{\succ}(x)v $.

$ \hfill \square $

 We call $ (\mathcal{A}, \ast) $ the associated q-generalized associative algebra of $ 
(\mathcal{A}, \prec, \succ) $ and $ (\mathcal{A}, \succ, \prec) $ is called a compatible q-generalized dendriform algebra structure on the q-generalized associative algebra $ (\mathcal{A}, \ast) $.

\begin{definition} \label{q-DA definition}
 Let $(\mathcal{A}, \succ, \prec)$ be a q-generalized dendriform algebra and $ V $ a vector space. 
Let $ l_{\succ}, r_{\succ}, l_{\prec}, r_{\prec} : \mathcal{A} \rightarrow gl(V) $ be four linear maps. 
Then, $(l_{\succ}, r_{\succ}, l_{\prec}, r_{\prec}, V)$ is called a \textbf{bimodule} of $ \mathcal{A} $
 if the following equations hold for any $ x, y \in \mathcal{A} $:
\beq
l_{\prec}(x \prec y) = q l_{\prec}(x)l_{\ast}(y), r_{\prec}(x)l_{\prec}(y)= q l_{\prec}(y)r_{\ast}(x), r_{\prec}(x)r_{\prec}(y) = q r_{\prec}(y \ast x), \eeq
\beq
l_{\prec}(x \succ y) = q l_{\succ}(x)l_{\prec}(y), r_{\prec}(x)l_{\succ}(y) = q l_{\succ}(y)r_{\prec}(x), r_{\prec}(x)r_{\succ}(y) = q r_{\succ}(y\prec x), 
\eeq
\beq
l_{\succ}(x \ast y) = q l_{\succ}(x)l_{\succ}(y), r_{\succ}(x)l_{\ast}(y)= q l_{\succ}(y)r_{\succ}(x), r_{\succ}(x)r_{\ast}(y) = q r_{\succ}(y \succ x),
\eeq
where $ x \ast y = x \succ y + x \prec y, l_{\ast} = l_{\succ} + l_{\prec}, r_{\ast} = r_{\succ} + r_{\prec} $.
\end{definition}

By a direct computation, $(l_{\succ}, r_{\succ}, l_{\prec}, r_{\prec}, V)$ is a bimodule of a q-generalized dendriform algebra $ (\mathcal{A},\succ, \prec) $ if and only if there exists a q-generalized dendriform algebra structure on the direct 
sum $ \mathcal{A} \oplus V $ of the underlying vector spaces of $\mathcal{A}$ given by
\beqs
(x + u) \succ (y + v) = x \succ y + l_{\succ}(x)v + r_{\succ}(y)u, \cr
(x + u) \prec (y + v) = x \prec y + l_{\prec}(x)v + r_{\prec}(y)u,
\eeqs
for all $ x, y \in \mathcal{A}, u, v \in V $. We denote it by $ \mathcal{A} \ltimes_{l_{\succ},r_{\succ},
l_{\prec}, r_{\prec}} V$.

\begin{proposition} \label{Proposition of bimodules}
Let  $(l_{\succ}, r_{\succ}, l_{\prec}, r_{\prec}, V)$ be a bimodule of a q-generalized dendriform algebra $(\mathcal{A}, 
\succ, \prec)$. Let $(\mathcal{A}, \ast)$ be the associated q-generalized associative algebra. Then, we have the following
 results: 
\benum
\item[(1)] Both $(l_{\succ}, r_{\prec}, V) $ and $ (l_{\succ} + l_{\prec}, r_{\succ} + r_{\prec}, V) $ are bimodules of $ (\mathcal{A}, \ast); $ 
\item[(2)] For any bimodule $(l, r, V)$ of $(\mathcal{A}, \ast)$, $(l, 0, 0, r, V)$ is a bimodule of
 $(\mathcal{A}, \succ, \prec); $ 
\item[(3)] Both $(l_{\succ} + l_{\prec}, 0, 0,  r_{\succ} + r_{\prec}, V)$ and $(l_{\succ}, 0, 0, r_{\prec}, V)$ are bimodules
 of $(\mathcal{A}, \succ, \prec);$ 
\item[(4)]  The q-generalized dendriform algebras $ \mathcal{A} \ltimes_{l_{\succ}, r_{\succ}, l_{\prec}, r_{\prec}} V $ and  $ \mathcal{A} \ltimes_{l_{\succ} +  l_{\prec} , 0, 0, r_{\succ} + r_{\prec}} V $ have the same associated
q-generalized associative algebra $ \mathcal{A} \ltimes_{l_{\succ} +  l_{\prec}, r_{\succ} + r_{\prec} } V$. 
\eenum
\end{proposition}
\textbf{Proof: }
By computing, we have:
\benum
\item $l_{\succ}(x\ast y)= q l_{\succ}(x)l_{\succ}(y),$ $r_{\prec}(x\ast y)= q^{-1}r_{\prec}(y)r_{\prec}(x),$ $l_{\succ}(x)r_{\prec}(y)= q^{-1}r_{\prec}(y)l_{\succ}(x).$ Then $(l_{\succ}, r_{\prec}, V)$ is bimodule of $(\mathcal{A}, \ast)$.\\
$(l_{\succ} + l_{\prec})(x\ast y)= (l_{\succ})(x\ast y) + (l_{\prec})(x\ast y)= q l_{\succ}(x)l_{\succ}(y) 
+ (l_{\prec})(x\prec y) + (l_{\prec})(x\succ y)= q l_{\succ}(x)l_{\succ}(y) + q l_{\prec}(x)l_{\ast}(y) + 
q l_{\succ}(x)l_{\prec}(y)= q (l_{\succ} + l_{\prec})(x)(l_{\succ} + l_{\prec})(y).$ By the same 
procedure, we establish the other relationships. Hence, $ (l_{\succ} + l_{\prec}, r_{\succ} + r_{\prec}, 
V) $ is bimodule of $(\mathcal{A}, \ast)$.
\item Setting $l_{\succ}=l, r_{\succ}=0, l_{\prec}=0, r_{\prec}=r$ in the deninition \ref{q-DA definition}  we have : $r(y\ast x)= q^{-1}r(x)r(y), l(y)r(x)= q^{-1}r(x)l(y), l(x\ast y)= ql(x)l(y).$ Then, $(l, 0, 0, r, V)$ is a bimodule of $(\mathcal{A}, \succ, \prec)$.
\item We reason similar to (2).
\item Using $(x + u)\ast (y + v)= (x + u)\succ (y + v)  + (x + u)\prec (y + v), \forall x, y\in \A, u, v\in V,$ with $(x + u) \succ (y + v) = x \succ y + l_{\succ}(x)v + r_{\succ}(y)u, 
(x + u) \prec (y + v) = x \prec y + l_{\prec}(x)v + r_{\prec}(y)u, $ we establish the result.
\eenum
$ \hfill \square $

\begin{example}
Let $(\mathcal{A}, \prec, \succ)$ be a dendriform algebra. Then,  
\beqs
(L_{\succ}, R_{\succ}, L_{\prec}, R_{\prec}, \mathcal{A}), (L_{\succ}, 0, 0, R_{\prec}, \mathcal{A}),
 (L_{\succ} + L_{\prec}, 0, 0,  R_{\succ} + R_{\prec}, \mathcal{A}) 
\eeqs
are bimodules of $(\mathcal{A}, \prec, \succ)$. 
\end{example}.

\begin{theorem}
Let $(\mathcal{A}, \succ_{\mathcal{A}}, \prec_{\mathcal{A}})$ and $(\mathcal{B}, \succ_{\mathcal{B}}, \prec_{\mathcal{B}})$
 be two q-generalized dendriform algebras. Suppose that there are linear maps 
$ l_{\succ_{\mathcal{A}}},   r_{\succ_{\mathcal{A}}},  l_{\prec_{\mathcal{A}}},  r_{\prec_{\mathcal{A}}} : \mathcal{A} \rightarrow gl(\mathcal{B})$ 
and $ l_{\succ_{\mathcal{B}}},   r_{\succ_{\mathcal{B}}},  l_{\prec_{\mathcal{B}}},  r_{\prec_{\mathcal{B}}} : \mathcal{A} \rightarrow gl(\mathcal{A})$ 
such that ( $ l_{\succ_{\mathcal{A}}},   r_{\succ_{\mathcal{A}}},  l_{\prec_{\mathcal{A}}},  r_{\prec_{\mathcal{A}}}, \mathcal{B}$) is a bimodule of $ \mathcal{A}
 $ and  ( $ l_{\succ_{\mathcal{B}}},   r_{\succ_{\mathcal{B}}},  l_{\prec_{\mathcal{B}}},  r_{\prec_{\mathcal{B}}}, \mathcal{A}$) is a bimodule  of 
$ \mathcal{B} $ 
and they satisfy the following equations :
\beq \label{eq35}
r_{\prec_{\mathcal{A}}}(x)(a \prec_{\mathcal{B}} b) = q a \prec_{\mathcal{B}}( r_{\mathcal{A}}(x)b) + q r_{\prec_{\mathcal{A}}}(l_{\mathcal{B}}(x)a),
\eeq
\beq \label{eq36}
l_{\prec_{\mathcal{A}}}(l_{\prec_{\mathcal{B}}}(x))b + (r_{\prec_{\mathcal{A}}}(x)a) \prec_{\mathcal{B}} b = q a \prec_{\mathcal{B}}
 (l_{\mathcal{A}}(x)b) + q r_{\prec_{\mathcal{A}}}(r_{\mathcal{B}}(b)x)a,
\eeq
\beq \label{eq37}
l_{\prec_{\mathcal{A}}}(x)(a \ast_{\mathcal{B}} b) = q^{-1} (l_{\prec_{\mathcal{A}}}(x)a) \prec_{\mathcal{B}} b +
q^{-1} l_{\prec_{\mathcal{A}}}(r_{\prec_{\mathcal{B}}}(a)x)b, 
\eeq
\beq \label{eq38}
r_{\prec_{\mathcal{A}}}(x)(a \succ_{\mathcal{B}} b) = q r_{\succ_{\mathcal{A}}}(l_{\prec_{\mathcal{B}}}(b)x)a + q a \succ_{\mathcal{B}} (r_{\prec_{\mathcal{A}}}(x)b), 
\eeq
\beq \label{eq39}
l_{\prec_{\mathcal{A}}}(l_{\succ_{\mathcal{B}}}(a)x)b + (r_{\succ_{\mathcal{A}}}(x)a) \prec_{\mathcal{B}} b = 
q a \succ_{B} (l_{\prec_{\mathcal{A}}}(x)b) + q r_{\succ_{\mathcal{A}}}(r_{\prec_{\mathcal{B}}}(b)x)a 
\eeq
\beq \label{eq40}
l_{\succ_{\mathcal{A}}}(x)(a \prec_{\mathcal{B}} b) = q^{-1} ( l_{\succ_{\mathcal{A}}}(x)a) \prec_{\mathcal{B}} b + q^{-1} l_{\prec_{\mathcal{A}}}(r_{\succ_{\mathcal{B}}}(a)x)b,
\eeq
\beq \label{eq41}
r_{\succ_{\mathcal{A}}}(x)(a \ast_{\mathcal{B}} b) = q a \succ_{\mathcal{B}} (r_{\succ_{\mathcal{A}}}(x)b) + q r_{\succ_{\mathcal{A}}}(l_{\succ_{\mathcal{B}}}(b)x)a, 
\eeq
\beq \label{eq42}
a \succ_{\mathcal{B}} (l_{\succ_{\mathcal{A}}}(x)b) + r_{\succ_{\mathcal{A}}}(r_{\succ_{\mathcal{B}}}(b)x)a 
= q^{-1} l_{\succ_{\mathcal{A}}}(l_{\mathcal{B}}(a)x)b + q^{-1} (r_{\mathcal{A}}(x)a) \succ_{\mathcal{B}} b, 
\eeq
\beq \label{eq43}
l_{\succ_{\mathcal{A}}}(x)(a \succ_{\mathcal{B}} b) = q^{-1} (l_{\mathcal{A}}(x)a) \succ_{\mathcal{B}} b + q^{-1} l_{\succ_{\mathcal{A}}}(r_{\mathcal{B}}(a)x)b,
\eeq
\beq \label{eq44} 
r_{\prec_{\mathcal{B}}}(a)(x \prec_{\mathcal{A}} y) = q x \prec_{\mathcal{A}} (r_{\mathcal{B}}(a)y) + q r_{\prec_{\mathcal{B}}}(l_{\mathcal{A}}(y)a)x,\eeq
\beq \label{eq45}
l_{\prec_{B}}(l_{\prec_{\mathcal{A}}}(x)a)y + (r_{\prec_{\mathcal{B}}}(a)x) \prec_{\mathcal{A}} y = 
q x \prec_{\mathcal{A}} (l_{\mathcal{B}}(a)y) + q r_{\prec_{\mathcal{B}}}(r_{\mathcal{A}}(y)a)x, 
\eeq
\beq \label{eq46}
l_{\prec_{\mathcal{B}}}(a)(x \ast_{\mathcal{A}} y) = q^{-1} (l_{\prec_{\mathcal{B}}}(a)x) \prec_{\mathcal{A}} y + q^{-1} l_{\prec_{\mathcal{B}}}(r_{\prec_{\mathcal{A}}}(x)a)y, 
\eeq
\beq \label{eq47}
r_{\prec_{\mathcal{B}}}(a)(x \succ_{\mathcal{A}} y) = q r_{\succ_{\mathcal{B}}}(l_{\prec_{\mathcal{B}}}(y)a)x + q x \succ_{\mathcal{A}} (r_{\prec_{\mathcal{B}}}(a)y), 
\eeq
\beq \label{eq48}
l_{\prec_{\mathcal{B}}}(l_{\succ_{\mathcal{A}}}(x)a)y + (r_{\succ_{\mathcal{B}}}(a)x) \prec_{\mathcal{A}} y = q x \succ_{\mathcal{A}} (l_{\prec_{\mathcal{B}}}(a)y) + q r_{\succ_{\mathcal{B}}}(r_{\prec_{\mathcal{A}}}(y)a)x, 
\eeq
\beq \label{eq49}
l_{\succ_{\mathcal{B}}}(a)(x \prec_{\mathcal{A}} y) = q^{-1} (l_{\succ_{\mathcal{B}}}(a)x) \prec_{\mathcal{A}} y + q^{-1} l_{\prec_{\mathcal{B}}}(r_{\succ_{\mathcal{A}}}(x)a)y, 
\eeq
\beq \label{eq50}
 r_{\succ_{\mathcal{B}}}(a)(x \ast_{\mathcal{A}} y) = q x \succ_{\mathcal{A}} (r_{\succ_{\mathcal{B}}}(a)y) + q r_{\succ_{\mathcal{B}}}(l_{\succ_{\mathcal{A}}}(y)a)x, 
 \eeq
 \beq \label{eq51}
 x \succ_{\mathcal{A}} (l_{\succ_{\mathcal{B}}}(a)y) + r_{\succ_{\mathcal{B}}}(r_{\succ_{\mathcal{A}}}(y)a)x =
q^{-1} l_{\succ_{\mathcal{B}}}(l_{\mathcal{A}}(x)a)y + q^{-1} (r_{B}(a)x) \succ_{\mathcal{A}} y, 
 \eeq
 \beq \label{eq52}
l_{\succ_{\mathcal{B}}}(a)(x \succ_{\mathcal{A}} y) = q^{-1} (l_{\mathcal{B}}(a)x) \succ_{\mathcal{A}} y + 
q^{-1} l_{\succ_{\mathcal{B}}}(r_{\mathcal{A}}(x)a)y
\eeq
for any $ x, y \in \mathcal{A}, a, b \in \mathcal{B} $ and $ l_{\mathcal{A}} = l_{\succ_{\mathcal{A}}} + 
 l_{\prec_{\mathcal{A}}}, r_{\mathcal{A}} =  r_{\succ_{\mathcal{A}}} +  r_{\prec_{\mathcal{A}}}, l_{\mathcal{B}} =
 l_{\succ_{\mathcal{B}}} +  l_{\prec_{\mathcal{B}}} , r_{\mathcal{B}} =  r_{\succ_{\mathcal{B}}} +  r_{\prec_{\mathcal{B}}} $.
 Then, there is a q-generalzed dendriform algebra structure on the direct sum $ \mathcal{A} \oplus \mathcal{B} $ of the underlying vector spaces of
 $ \mathcal{A} $ and $ \mathcal{B} $ given by
\beqs
(x + a) \succ ( y + b ) = (x \succ_{\mathcal{A}} y + r_{\succ_{\mathcal{B}}}(b)x + l_{\succ_{\mathcal{B}}}(a)y) +
 (l_{\succ_{\mathcal{A}}}(x)b + r_{\succ_{\mathcal{A}}}(y)a + a \succ_{\mathcal{B}} b ), \cr \cr
(x + a) \prec ( y + b ) = (x \prec_{\mathcal{A}} y + r_{\prec_{\mathcal{B}}}(b)x + l_{\prec_{\mathcal{B}}}(a)y) + 
(l_{\prec_{\mathcal{A}}}(x)b + r_{\prec_{\mathcal{A}}}(y)a + a \prec_{\mathcal{B}} b )
\eeqs
for any $ x, y \in \mathcal{A}, a, b \in \mathcal{B} $. Let $ \mathcal{A} \bowtie^{l_{\succ_{\mathcal{A}}}, r_{\succ_{\mathcal{A}}}, 
l_{\prec_{\mathcal{A}}}, r_{\prec_{\mathcal{A}}}}_{l_{\succ_{\mathcal{B}}}, r_{\succ_{\mathcal{B}}}, l_{\prec_{\mathcal{B}}}, 
r_{\prec_{\mathcal{B}}}} \mathcal{B} $ or simply $ \mathcal{A} \bowtie \mathcal{B} $  denote this q- generaized dendriform algebra. On the other hand, every q-generlized dendriform algebra which is the direct sum of the underlying vector spaces of two subalgebra can be obtained in  this way.  
\end{theorem}
\textbf{Proof: }
The Proof  is obtained in a similar way as for Theorem \ref{matched pair theorem}.

$ \hfill \square $

\begin{definition}
Let $ (\mathcal{A}, \succ_{\mathcal{A}}, \prec_{\mathcal{A}}) $ and $  (\mathcal{B}, \succ_{\mathcal{B}}, \prec_{\mathcal{B}}) $ 
be two q-generalized dendriform algebras. Suppose that there are linear maps 
$ l_{\succ_{\mathcal{A}}}, r_{\succ_{\mathcal{A}}}, l_{\prec_{\mathcal{A}}}, r_{\prec_{\mathcal{A}}} : \mathcal{A} \rightarrow gl(\mathcal{B}) $ 
and $ l_{\succ_{\mathcal{B}}}, r_{\succ_{\mathcal{B}}}, l_{\prec_{\mathcal{B}}}, r_{\prec_{\mathcal{B}}} : \mathcal{B} \rightarrow gl(\mathcal{A}) $
 such that $(l_{\succ_{\mathcal{A}}}, r_{\succ_{\mathcal{A}}}, l_{\prec_{\mathcal{A}}}, r_{\prec_{\mathcal{A}}})$ is a bimodule of $ \mathcal{A} $ 
and $(l_{\succ_{\mathcal{B}}}, r_{\succ_{\mathcal{B}}}, l_{\prec_{\mathcal{B}}}, r_{\prec_{\mathcal{B}}})$ is a bimodule of $ \mathcal{B} $. 
If Eqs.\ref{eq35} - \ref{eq52} are satisfied, then $(\mathcal{A}, \mathcal{B},l_{\succ_{\mathcal{A}}}, 
r_{\succ_{\mathcal{A}}}, l_{\prec_{\mathcal{A}}}, r_{\prec_{\mathcal{A}}},l_{\succ_{\mathcal{B}}}, r_{\succ_{\mathcal{B}}}, l_{\prec_{\mathcal{B}}},
 r_{\prec_{\mathcal{B}}})$ is called a \textbf{matched pair of q-generalized dendriform algebras}.   
\end{definition}

\begin{remark}
Obviously $ \mathcal{B} $ is an ideal of $ \mathcal{A} \bowtie \mathcal{B} $ if and only if
 $ l_{\succ_{\mathcal{B}}}= r_{\succ_{\mathcal{B}}}= l_{\prec_{\mathcal{B}}}= r_{\prec_{\mathcal{B}}} =0 $. If $ \mathcal{B} $ 
is a trivial ideal, then 
$ \mathcal{A} \bowtie^{l_{\succ_{\mathcal{A}}}, r_{\succ_{\mathcal{A}}}, l_{\prec_{\mathcal{A}}}, r_{\prec_{\mathcal{A}}}}_{0, 0, 0, 0} 
\mathcal{B} \cong \mathcal{A} \ltimes_{l_{\succ_{\mathcal{A}}}, r_{\succ_{\mathcal{A}}}, l_{\prec_{\mathcal{A}}}, r_{\prec_{\mathcal{A}}}} 
\mathcal{B} $. 
\end{remark}
\begin{corollary}\label{corollary 3.2.7}
Let $ (\mathcal{A}, \mathcal{B}, l_{\succ_{\mathcal{A}}}, r_{\succ_{\mathcal{A}}}, l_{\prec_{\mathcal{A}}}, r_{\prec_{\mathcal{A}}},
 l_{\succ_{\mathcal{B}}}, r_{\succ_{\mathcal{B}}}, l_{\prec_{\mathcal{B}}}, r_{\prec_{\mathcal{B}}}) $ be a matched pair of q-generalized dendriform algebras. 
Then, $(\mathcal{A}, \mathcal{B}, l_{\succ_{\mathcal{A}}} + l_{\prec_{\mathcal{A}}}, r_{\succ_{\mathcal{A}}} + r_{\prec_{\mathcal{A}}},  
l_{\succ_{\mathcal{B}}} + l_{\prec_{\mathcal{B}}},  r_{\succ_{\mathcal{B}}} + r_{\prec_{\mathcal{B}}})$ is a matched pair of the associated q-generalized associative algebras $ (\mathcal{A}, \ast_{\mathcal{A}}) $ and  $(\mathcal{B}, \ast_{\mathcal{B}}) $.
\end{corollary}
\textbf{Proof: }
In fact, the associated q-generalized associative algebra $(\mathcal{A} \bowtie \mathcal{B}, \ast)$ is exactly the q-generalized associative algebra obtained 
from the matched pair $(\mathcal{A}, \mathcal{B}, l_{\mathcal{A}}, r_{\mathcal{A}}, l_{\mathcal{B}}, r_{\mathcal{B}})$ of q-generalized associative algebras:
\beqs
(x + a)\ast (y + b) = x\ast_{\mathcal{A}} y + l_{\mathcal{B}}(a)y + r_{\mathcal{B}}(b)x + a \ast_{\mathcal{B}} b +  l_{\mathcal{A}}(x)b + 
r_{\mathcal{A}}(y)a
\eeqs 
for all $ x, y \in \mathcal{A}, a, b \in \mathcal{B} $, where $ l_{\mathcal{A}} = l_{\succ_{\mathcal{A}}} + l_{\prec_{\mathcal{A}}}, r_{\mathcal{A}}
= r_{\succ_{\mathcal{A}}} + r_{\prec_{\mathcal{A}}}, l_{\mathcal{B}} = l_{\succ_{\mathcal{B}}} + l_{\prec_{\mathcal{B}}}, r_{\mathcal{B}}= 
r_{\succ_{\mathcal{B}}} + r_{\prec_{\mathcal{B}}}  $.

$ \hfill \square $

\begin{proposition}
Let ($l_{\succ}, r_{\succ}, l_{\prec}, r_{\prec}, V$) be a bimodule of a q-generalized dendriform algebra $(\mathcal{A}, \succ, \prec)$. Let $(\mathcal{A}, \ast)$ be the associated q-generalized hom-associative algebra.
 Let $ l^{\ast}_{\succ},r^{\ast}_{\succ},l^{\ast}_{\prec}, r^{\ast}_{\prec} : \mathcal{A} 
\rightarrow gl(V^{\ast}) $ be the linear maps given by 
\beqs
\langle l^{\ast}_{\succ}(x)a^{\ast}, y \rangle = \langle l_{\succ}(x)y, a^{\ast} \rangle, 
\langle r^{\ast}_{\succ}(x)a^{\ast}, y \rangle = \langle r_{\succ}(x)y, a^{\ast} \rangle, \cr
\langle l^{\ast}_{\prec}(x)a^{\ast}, y \rangle = \langle l_{\prec}(x)y, a^{\ast} \rangle, 
\langle r^{\ast}_{\prec}(x)a^{\ast}, y \rangle = \langle r_{\prec}(x)y, a^{\ast} \rangle.
\eeqs
Then,
\benum
\item[(1)]
  $(q^{-2}(r^{\ast}_{\succ} +  r^{\ast}_{\prec}), -q^{2}l^{\ast}_{\prec}, -q^{-2}r^{\ast}_{\succ}, q^{2}(l^{\ast}_{\succ} + l^{\ast}_{\prec}), V^{\ast})$ is a bimodule of $(\mathcal{A},\succ,\prec);$ 
\item[(2)] $(q^{-2}(r^{\ast}_{\succ} +  r^{\ast}_{\prec}), 0, 0, q^{2}(l^{\ast}_{\succ} + l^{\ast}_{\prec}), V^{\ast})$
 and $(q^{-2}r^{\ast}_{\prec}, 0, 0, q^{2}l^{\ast}_{\succ}, V^{\ast})$ are bimodules of $(\mathcal{A},\succ,\prec); $ 
\item[(3)] $(q^{-2}(r^{\ast}_{\succ} +  r^{\ast}_{\prec}), q^{2}(l^{\ast}_{\succ} + l^{\ast}_{\prec}), V^{\ast})$ and 
$(q^{-2}r^{\ast}_{\prec}, q^{2}l^{\ast}_{\succ}, V^{\ast})$ are bimodules of $(\mathcal{A}, \ast); $ 
\item[(4)] The q-generalized dendriform algebras  \beqs \mathcal{A} \ltimes_{q^{-2}(r^{\ast}_{\succ} + r^{\ast}_{\prec}), -q^{2}l^{\ast}_{\prec}, -q^{-2}r^{\ast}_{\succ}, q^{2}(l^{\ast}_{\succ} + l^{\ast}_{\prec})} V^{\ast}  \mbox{ and } \mathcal{A} \ltimes_{q^{-2}r^{\ast}_{\prec}, 0, 0, q^{2}l^{\ast}_{\succ}} V^{\ast}\eeqs have the same q-generalized associative algebra 
$\mathcal{A} \ltimes_{q^{-2}r^{\ast}_{\prec}, q^{2}l^{\ast}_{\succ}} V^{\ast} $.
\eenum
\end{proposition}
\textbf{Proof: }
Show that $(q^{-2}(r^{\ast}_{\succ} +  r^{\ast}_{\prec}), -q^{2}l^{\ast}_{\prec}, -q^{-2}r^{\ast}_{\succ}, q^{2}(l^{\ast}_{\succ} + l^{\ast}_{\prec}), V^{\ast})$ is a bimodule of $(\mathcal{A},\succ,\prec)$. Let $ x, y \in  \mathcal{A}, u^{\ast} \in V^{\ast}, v \in V $, we have :
\begin{enumerate}
\item[(i)]
\beqs
 \langle -q^{-2}r^{\ast}_{\succ}(x\prec y)u^{\ast}, v \rangle = \langle -q^{-2}r_{\succ}(x\prec y)v, u^{\ast} \rangle = \langle -q^{-3}r_{\prec}(y)r_{\succ}(x)v, u^{\ast}\rangle\cr 
= \langle q(-q^{-2}r^{\ast}_{\succ})(x)(q^{-2}r^{\ast}_{\prec})(y)u^{\ast}, v\rangle 
\eeqs
 leading to $(-q^{-2}r^{\ast}_{\succ})(x\prec y)u^{\ast}= q(-q^{-2}r^{\ast}_{\succ})(x)(q^{-2}r^{\ast}_{\prec})(y)u^{\ast}$;

\item[(ii)]
\beqs
 \langle (q^{2}l^{\ast}_{\ast}(x))(-q^{-2}r^{\ast}_{\succ}(y))u^{\ast}, v\rangle = \langle -r_{\succ}(y)l_{\ast}(x)v, u^{\ast} \rangle\cr
= \langle -ql_{\succ}(x)r_{\succ}(y)v, u^{\ast}\rangle
                                        = \langle q(-q^{2}r^{\ast}_{\succ}(y))(q^{2}l^{\ast}_{\succ}(x))(u^{\ast}, v \rangle   
\eeqs
 giving  $(q^{2}l^{\ast}_{\ast}(x))(-q^{-2}r^{\ast}_{\succ}(y))u^{\ast} = q(-q^{-2}r^{\ast}_{\succ}(y))(q^{2}l^{\ast}_{\succ}(x))u^{\ast}$;

\item[(iii)]
\beqs
\langle (q^{2}l^{\ast}_{\ast})(x)(q^{2}l^{\ast}_{\ast})(y)u^{\ast}, v\rangle                                               = \langle q^{4}l_{\ast}(y)l_{\ast}(x)v, u^{\ast}\rangle\cr
=\langle q^{4}l_{\prec}(y)l_{\ast}(x)v, u^{\ast}\rangle + \langle q^{4}l_{\succ}(y)l_{\succ}(x)v, u^{\ast}\rangle + \langle q^{4}l_{\succ}(y)l_{\prec}(x)v, u^{\ast}\rangle\cr
   = \langle q^{3}l_{\prec}(y\prec x)v, u^{\ast} \rangle + \langle q^{3}l_{\succ}(y\ast x)v, u^{\ast} \rangle + \langle q^{3}l_{\prec}(y\succ x)v, u^{\ast} \rangle\cr
                =\langle q^{3}l_{\ast}(y\ast x)v, u^{\ast}\rangle= \langle q^{3}l^{\ast}_{\ast}(y\ast x)u^{\ast}, v\rangle
\eeqs
 providing  that  $(q^{2}l^{\ast}_{\ast})(x)(q^{2}l^{\ast}_{\ast})(y)u^{\ast}= q(q^{2}l^{\ast}_{\ast}(y\ast x))u^{\ast}$;
 \item[(iv)]
 \beqs
 \langle -q^{-2}r^{\ast}_{\succ}(x\succ y)u^{\ast}, v\rangle = \langle -q^{-2}r_{\succ}(x\succ y)v, u^{\ast} \rangle = \langle -q^{-3}r_{\succ}(y)r_{\ast}(x)v, u^{\ast}\rangle\cr 
= \langle q(q^{-2}r^{\ast}_{\ast})(x)(-q^{-2}r^{\ast}_{\succ})(y)u^{\ast}, v\rangle 
\eeqs
 hence $(-q^{-2}r^{\ast}_{\succ}(x\succ y)u^{\ast}= q(q^{-2}r^{\ast}_{\ast})(x)(-q^{-2}r^{\ast}_{\succ})(y)u^{\ast}$;
 \item[(v)]
\beqs
\langle (q^{2}l^{\ast}_{\ast})(x)(q^{-2}r^{\ast}_{\ast})(y)u^{\ast}, v\rangle                                               = \langle r_{\ast}(y)l_{\ast}(x)v, u^{\ast}\rangle\cr
=\langle r_{\succ}(y)l_{\ast}(x)v, u^{\ast}\rangle + \langle r_{\prec}(y)l_{\succ}(x)v, u^{\ast}\rangle + \langle r_{\prec}(y)l_{\prec}(x)v, u^{\ast}\rangle\cr
   = \langle ql_{\succ}(x)r_{\succ}(y)v, u^{\ast} \rangle + \langle ql_{\succ}(x)r_{\prec}(y)v, u^{\ast}\rangle + \langle ql_{\prec}(x)r_{\ast}(y)v, u^{\ast} \rangle\cr
                =\langle ql_{\ast}(x)r_{\ast}(y)v, u^{\ast}\rangle= \langle q(q^{-2}r^{\ast}_{\ast}(y))(q^{2}l^{\ast}_{\ast}(x))u^{\ast}, v\rangle
\eeqs
 providing  that  $(q^{2}l^{\ast}_{\ast})(x)(q^{-2}r^{\ast}_{\ast})(y)u^{\ast}= q(q^{-2}r^{\ast}_{\ast}(y))(q^{2}l^{\ast}_{\ast}(x))u^{\ast}$;
  \item[(vi)]
 \beqs
 \langle -q^{3}l^{\ast}_{\prec}(y\prec x)u^{\ast}, v\rangle = \langle -q^{3}l_{\prec}(y\prec x)v, u^{\ast}\rangle = \langle -q^{-4}l_{\prec}(y)l_{\ast}(x)v, u^{\ast}\rangle\cr 
= \langle(q^{2}l^{\ast}_{\ast})(x)(-q^{2}l^{\ast}_{\prec})(y)u^{\ast}, v\rangle 
\eeqs
then $(q^{2}l^{\ast}_{\ast})(x)(-q^{2}l^{\ast}_{\prec})(y)u^{\ast}= q(-q^{2}l^{\ast}_{\prec})(y\prec x)u^{\ast};$
\item[(vii)]
\beqs
 \langle q^{-2}r^{\ast}_{\ast}(x\ast y)u^{\ast}, v \rangle
  = \langle q^{-2}r_{\ast}(x\ast y)v, u^{\ast}\rangle\cr
   = \langle q^{-2}r_{\succ}(x\prec y)v + q^{-2}r_{\succ}(x\succ y)v + q^{-2}r_{\prec}(x\ast y)v, u^{\ast}\rangle\cr
\langle q^{-2}q^{-1}r_{\prec}(y)r_{\succ}(x)v + q^{-2}q^{-1}r_{\succ}(y)r_{\ast}(x)v + q^{-2}q^{-1}r_{\prec}(y)r_{\prec}(x)v, u^{\ast}\rangle\cr 
=\langle q^{-3}r_{\ast}(y)r_{\ast}(x)v, u^{\ast}\rangle= \langle q^{-3}r^{\ast}_{\ast}(y)r^{\ast}_{\ast}(x)u^{\ast}, v\rangle 
\eeqs
 leading to $(q^{-2}r^{\ast}_{\ast})(x\ast y)u^{\ast}= q(q^{-2}r^{\ast}_{\ast})(x)(q^{-2}r^{\ast}_{\ast})(y)u^{\ast}$;
 \item[(viii)]
\beqs
 \langle (-q^{2}l^{\ast}_{\prec}(x))(q^{-2}r^{\ast}_{\prec}(y))u^{\ast}, v\rangle = \langle -r_{\prec}(y)l_{\prec}(x)v, u^{\ast} \rangle\cr
= \langle ql_{\prec}(x)r_{\ast}(y)v, u^{\ast}\rangle
= \langle q(q^{2}r^{\ast}_{\ast}(y))(-q^{2}l^{\ast}_{\prec}(x))(u^{\ast}, v \rangle   
\eeqs
 giving  $(-q^{2}l^{\ast}_{\prec}(x))(q^{-2}r^{\ast}_{\prec}(y))u^{\ast} = q(q^{-2}r^{\ast}_{\ast}(y))(-q^{2}l^{\ast}_{\prec}(x))u^{\ast}$;
  \item[(ix)]
\beqs
 \langle (-q^{3}l^{\ast}_{\prec}(y\succ x))u^{\ast}, v\rangle = \langle (-q^{3}l^{\ast}_{\prec}(y\succ x))v, u^{\ast}\rangle\cr
= \langle -q^{4}l_{\succ}(y)l_{\prec}(x)v, u^{\ast}\rangle
= \langle (-q^{2}l^{\ast}_{\prec}(x))(q^{2}l^{\ast}_{\succ}(y))u^{\ast}, v\rangle   
\eeqs
 giving  $(-q^{2}l^{\ast}_{\prec}(x))(q^{2}l^{\ast}_{\succ}(y))u^{\ast} = q(-q^{2}l^{\ast}_{\prec}(y\succ x))u^{\ast}$;
\end{enumerate}
Similar, using the Propostion \ref{Proposition of bimodules} we show the other results.

$ \hfill \square $

\begin{example}
Let $(\mathcal{A}, \prec, \succ)$ be a q-generalized dendriform algebra. Then,  
\beqs
 (q^{-2}(R^{\ast}_{\succ} + R^{\ast}_{\prec}), -q^{-2}L^{\ast}_{\prec}, -q^{-2}R^{\ast}_{\succ},  q^{2}(L^{\ast}_{\succ} + L^{\ast}_{\prec}), \mathcal{A}^{\ast}),\cr  (q^{-2}R^{\ast}_{\prec}, 0, 0, q^{2}L^{\ast}_{\succ}, \mathcal{A}^{\ast}), (q^{-2}(R^{\ast}_{\succ} 
+ R^{\ast}_{\prec}), 0, 0, q^{2}(L^{\ast}_{\succ} + L^{\ast}_{\prec}), \mathcal{A}^{\ast})
\eeqs
are bimodules of $(\mathcal{A}, \succ, \prec)$  too. There are two compatible  q-generalized dendriform algebra structures, 
$
\mathcal{A} \ltimes_{R^{\ast}_{\succ} + R^{\ast}_{\prec}, -L^{\ast}_{\prec}, -R^{\ast}_{\succ},  
L^{\ast}_{\succ} + L^{\ast}_{\prec}} \mathcal{A}^{\ast} \mbox { and } \mathcal{A} \ltimes_{R^{\ast}_{\succ} + R^{\ast}_{\prec}, 0, 0,  L^{\ast}_{\succ} + L^{\ast}_{\prec}} \mathcal{A}^{\ast},
$
on the same associative algebra $  \mathcal{A} \ltimes_{ R^{\ast}_{\prec},  L^{\ast}_{\succ}}
 \mathcal{A}^{\ast} $.
\end{example}.

\section{Double construction of sympletic  antiassociative algebras}
\subsection{$ \mathcal{O} $-operators and antidendriform algebras } 
\begin{definition}
Let $(\mathcal{A}, \cdot)$ be an antiassociative algebra and $(l, r, V)$ a bimodule. A linear map $ T : V \rightarrow \mathcal{A} $ 
is called an \textbf{$ \mathcal{O} $-operator associated} to $(l, r, V)$,  if $ T $ satisfies
\beqs
T(u)\cdot T(v) = T(l(T(u))v + r(T(v))u) \mbox { for all } u, v \in V.
\eeqs
\end{definition}
\begin{example}
Let $(\mathcal{A}, \cdot)$ be an antiassociative algebra. Then,
 the identity map $ \id $ is an $ \mathcal{O} $-operator associated to the bimodule $(L,0)$ or $(0,R)$.
\end{example}

\begin{example}
Let $(\mathcal{A}, \cdot)$ be an antiassociative algebra. 
A linear map $ \tau : \mathcal{A} \rightarrow \mathcal{A} $ is called a \textbf{Rota-Baxter operator} on $ \mathcal{A} $ of weight zero if $ \tau $ 
satisfies
\beqs
\tau(x)\cdot \tau(y) = \tau(\tau(x)\cdot y + x\cdot \tau(y)) \mbox { for all } x, y \in \mathcal{A}.
\eeqs
In fact, a Rota-Baxter operator on $ \mathcal{A} $ is just an $ \mathcal{O} $-operator associated to the bimodule $(L, R)$.
\end{example}
\begin{theorem}\label{theo3.1.2}
 Let $ \mathcal{A} $ be an antiassociative algebra and $ (l, r, V) $ a bimodule. 
Let $ T : V \rightarrow \mathcal{A} $ be an $ \mathcal{O} $-operator associated to $ (l, r, V) $. Then, there exists an antidendriform 
algebra structure on $ V $ given by
\beqs
 u \succ v = l(T(u))v ,  u \prec v = r(T(v))u 
\eeqs
for all $ u, v \in V $. So, there is an associated antiassociative algebra structure on $ V $ given by the equation (\ref{antiassociative product})  and $ T $ 
is a homomorphism of antiassociative algebras. Moreover, $ T(V) = \lbrace { T(v) \setminus v \in V }  \rbrace  \subset \mathcal{A} $ is an antiassociative 
subalgebra of $ \mathcal{A} $ and there is an induced antidendriform algebra structure on $ T(V) $ given by
\beq \label{eq33}
T(u) \succ T(v) = T(u \succ v), T(u) \prec T(v) = T(u \prec v)
\eeq
for all $ u, v \in V $. Its corresponding associated antiassociative algebra structure on $ T(V) $ given by the equation (\ref{antiassociative product}) is 
just the antiassociative subalgebra structure of $ \mathcal{A} $ and $ T $ is a homomorphism of antidendriform algebras.
\end{theorem}
\textbf{Proof: }
For any $x, y, z\in V,$ we have
\beqs
&&(x\succ y)\prec z + x\succ(y\prec z)=l(T(x))y\prec z + x\succ r(T(z)y)\cr
&&=r(T(z))l(T(x))y+l(T(x))r(T(z))y=-l((T(x)))r(T(z))y + l((T(x)))r(T(z))y\cr
&&=0.
\eeqs
The other equalities are checked similarly.

$ \hfill \square $
\begin{corollary}\label{corollary 3.1.3}
 Let $ (\mathcal{A}, \ast) $ be an antiassociative algebra. There is a compatible antidendriform algebra structure on $ \mathcal{A} $ 
if and only if there exists an invertible $ \mathcal{O} $-operator of $ (\mathcal{A}, \ast) $.
\end{corollary}
\textbf{Proof: }   
In fact, if $ T $ is an invertible $ \mathcal{O}- $operator associated to a bimodule $(l, r, V)$, then 
the compatible antidendriform algebra structure on $ \mathcal{A} $ is given by 
\beqs
x \succ y = T(l(x)T^{-1}(y)), x \prec y = T(r(y)T^{-1}(x)) \mbox { for all } x, y \in \mathcal{A}.
\eeqs
Conversely, let $(\mathcal{A}, \succ, \prec)$ be a antidendriform algebra and $(\mathcal{A}, \ast)$
 the associated antiassociative algebra. Then, the identity map $ \id $ is an $ \mathcal{O}- $operator 
associated to the bimodule $(L_{\succ}, R_{\prec})$ of $(\mathcal{A}, \ast)$. 

$ \cqfd$
\subsection{Double constructions of sympletic  antiassociative algebras }
\begin{definition}
Let $\A$ be an antiassociative algebra. We say that $(\A, \omega)$ is a sympletic antiassociative algebra if $\omega$ is a non-degenerate skew-symmetric bilinear form on $\A$ such that the following identity satisfied:
\beq
\omega(xy, z) + \omega(yz, x) + \omega(zx, y)=0\ \mbox{(invariance condition),}
\eeq 
for all $x, y, z\in\A.$
\end{definition}
\begin{theorem}\label{theo4.1.1}
Let $(\mathcal{A}, \ast)$ be an antiassociative algebra and let $ \omega $ be a non-degenerate skew-symmetric bilinear form. 
Then,  there exists a compatible antidendriform algebra structure $ \succ, \prec $ on $  \mathcal{A} $ given by 
\beq \label{eq53}
\omega(x \succ y, z) = \omega(y, z \ast x), \ \ \ \omega(x \succ y, z) = \omega(x, y \ast z) \mbox { for all } x, y \in \mathcal{A}.
\eeq 
\end{theorem}
\textbf{Proof: }
Define a linear map $ T : \mathcal{A} \rightarrow \mathcal{A}^{\ast} $ by $ \langle T(x), y \rangle = \omega(x, y) $ for all $ x, y \in \mathcal{A} $. 
Then, $ T $ is invertible and $ T^{-1} $ is an $  \mathcal{O}-$operator of the antiassociative algebra $(\mathcal{A}, \ast)$ associated to the 
bimodule $(R^{\ast}_{\ast}, L^{\ast}_{\ast})$. By Corollary \ref{corollary 3.1.3}, there is a compatible antidendriform algebra structure 
$ \succ, \prec $ on $(\mathcal{A}, \ast)$ given by
\beqs
x \succ y = T^{-1}R^{\ast}_{\ast}(x)T(y), \ \ x \prec y =  T^{-1}L^{\ast}_{\ast}(y)T(x)
\eeqs
for all $ x, y \in \mathcal{A} $, which gives exactly the equation (\ref{eq53}).
$ \hfill \square $

\begin{definition}
We call $(\mathcal{A},  \omega)$ a \textbf{double construction of sympletic antiassociative algebra} associated to $\mathcal{A}_1$ and ${\mathcal A}_1^*$ if it satisfies the conditions 
\benum
 \item[(1)] $ \mathcal{A} = \mathcal{A}_{1}
\oplus \mathcal{A}^{\ast}_{1} $ as the direct sum of vector
spaces; 
\item[(2)] $\mathcal{A}_1$ and ${\mathcal A}_1^*$ are antiassociative subalgebras of $
(\mathcal{A}$;
\item[(3)] $\omega$ is the natural non-degenerate antisymmetric invariant bilinear form on $\mathcal{A}_{1} \oplus \mathcal{A}^{\ast}_{1}$
given by \beq \label{sympletic form}
 \omega(x + a^{\ast}, y +  b^{\ast})&=& -\langle x, b^{\ast} \rangle + \langle a^{\ast}, y \rangle, 
\eeq
for all  $x, y \in \mathcal{A}_{1}, a^{\ast}, b^{\ast} \in \mathcal{A}^{\ast}_{1}$
where $ \langle  , \rangle $ is the natural pair between the vector space $ \mathcal{A}_{1} $ and its dual space  $ \mathcal{A}^{\ast}_{1} $.
\eenum
\end{definition} 

Let $(\mathcal{A}, \ast_{\mathcal{A}})$ be an
 antiassociative algebra and suppose that there is an antiassociative algebra structure $ \ast_{\mathcal{A}^{\ast}} $ on its dual 
space $ \mathcal{A}^{\ast} $. We construct an antiassociative algebra structure on the direct sum
 $ \mathcal{A} \oplus \mathcal{A}^{\ast} $ of the underlying vector spaces of $ \mathcal{A} $ and $ \mathcal{A}^{\ast} $ such
 that both $ \mathcal{A} $ and $ \mathcal{A}^{\ast} $ are subalgebras and 
 the antisymmetric bilinear form on $ \mathcal{A} \oplus \mathcal{A}^{\ast} $ given by Eq.(\ref{sympletic form}) is invariant on $ \mathcal{A} \oplus \mathcal{A}^{\ast} $. Such a construction is called double construction of sympletic antiassociative algebra associated to $ (\mathcal{A}, \ast_{\mathcal{A}}) $ and $ (\mathcal{A}^{\ast}, \ast_{\mathcal{A}^{\ast}}) $ and denoted by $ (T(\mathcal{A}) = \mathcal{A} \bowtie \mathcal{A}^{\ast}, \omega)  $.
\begin{corollary}
Let $(T(\mathcal{A}) = \mathcal{A} \bowtie \mathcal{A}^{\ast}, \omega)$ be a double construction of sympletic antiassociative algebra. Then, 
 there exists a compatible antidendriform algebra structure $ \succ, \prec $ on $ T(\mathcal{A}) $ 
defined by the equation (\ref{eq53}).

  Moreover, $ \mathcal{A} $ and $ \mathcal{A}^{\ast} $, endowed with this product,
 are antidendriform subalgebras. 
\end{corollary}
\textbf{Proof: }
The first  half follows from Theorem \ref{theo4.1.1}. 
Let $ x, y \in \mathcal{A} $. Set $ x \succ y = a + b^{\ast} $, where 
$ a \in \mathcal{A}, b^{\ast} \in \mathcal{A}^{\ast} $. Since $ \mathcal{A} $ is an antiassociative subalgebra of $T(\mathcal{A})$ 
and $ \omega(\mathcal{A}, \mathcal{A}) = \omega(\mathcal{A}^{\ast}, \mathcal{A}^{\ast}) = 0 $, we have
\beqs
\omega(b^{\ast}, \mathcal{A}^{\ast}) = \omega(b^{\ast}, \mathcal{A}) = \omega(x \succ y, \mathcal{A}) = \omega(y, \mathcal{A} \ast x) = 0.
\eeqs
Therefore, 
 $ b^{\ast} = 0 $ due to the non-dependence of $ \omega $. Hence, $ x \succ y = a \in \mathcal{A}$. Similarly, $ x \prec y  \in \mathcal{A} $.
 Thus, $ \mathcal{A} $ is an antidendriform subalgebra of $ T(\mathcal{A}) $ with the product $ \prec, \succ $.
 By symmetry of $ \mathcal{A} $, $ \mathcal{A}^{\ast} $ is also an antidendriform subalgebra.

$ \hfill \square $
\begin{definition}
Let $(T(\mathcal{A}_{1}) = \mathcal{A}_{1} \bowtie \mathcal{A}^{\ast}_{1}, \omega_{1})$ and $(T(\mathcal{A}_{2}) = 
\mathcal{A}_{2} \bowtie \mathcal{A}^{\ast}_{2}, \omega_{2})$ be two double constructions of sympletic antiassociative algebra. 
They are \textbf{isomorphic} if there exists an isomorphism of antiassociative algebras $ \varphi : T(\mathcal{A}_{1}) \rightarrow T(\mathcal{A}_{2}) $ 
satisfying the conditions
\beq \label{eq54}
\varphi (\mathcal{A}_{1}) = \mathcal{A}_{2}, \ \ \varphi (\mathcal{A}^{\ast}_{1}) = \mathcal{A}^{\ast}_{2}, \ \ \omega_{1}(x, y)
 = \varphi^{\ast}  \omega_{2}(x, y) = \omega_{2}(\varphi(x), \varphi(y))  
\eeq
for all $ x, y \in \mathcal{A}_{1} $.
\end{definition}
\begin{proposition}
Two double constructions of sympletic antiassociative algebra  $(T(\mathcal{A}_{1}) = \mathcal{A}_{1} \bowtie 
\mathcal{A}^{\ast}_{1}, \omega_{1})$ and $(T(\mathcal{A}_{2}) = \mathcal{A}_{2} \bowtie
 \mathcal{A}^{\ast}_{2}, \omega_{2})$ are isomorphic if and  only if there exists a antidendriform algebra isomorphism $ \varphi : T(\mathcal{A}_{1}) \rightarrow T(\mathcal{A}_{1}) $ satisfying the equation (\ref{eq54}), where the antidendriform 
algebra structures on $ T(\mathcal{A}_{1}) $ and $  T(\mathcal{A}_{2}) $  and given by the equation (\ref{eq53}), respectively.  
\end{proposition}
\textbf{Proof: }
This is straightforward.
$ \hfill \square $

\begin{theorem}\label{theo4.1.5},
Let $ (\mathcal{A}, \succ_{\mathcal{A}}, \prec_{\mathcal{A}}) $ be an antidendriform algebra and $ (\mathcal{A}, \ast_{\mathcal{A}}) $ 
the associated antiassociative algebra. Suppose that there is an antidendriform algebra  structure 
$ " \succ_{\mathcal{A}^{\ast}}, \prec_{\mathcal{A}^{\ast}} " $ on its dual space $  \mathcal{A}^{\ast} $ and
 $ (\mathcal{A}^{\ast}, \ast_{\mathcal{A}^{\ast}}) $ is the associated antiassociative algebra. Then, there exists a double construction 
of sympletic antiassociative algebra associated to $(\mathcal{A}, \ast_{\mathcal{A}})$ and $(\mathcal{A}, \ast_{\mathcal{A}^{\ast}})$
if and only if 
$(\mathcal{A}, \mathcal{A}^{\ast}, R^{\ast}_{\prec_{\mathcal{A}}}, L^{\ast}_{\succ_{\mathcal{A}}}, R^{\ast}_{\prec_{\mathcal{A}^{\ast}}},
L^{\ast}_{\succ_{\mathcal{A}^{\ast}}})$ is a matched pair of antiassociative algebras. Moreover, every double construction of the sympletic antiassociative algebra can be obtained in this way.
\end{theorem}
\textbf{Proof: }
If  $(\mathcal{A}, \mathcal{A}^{\ast}, R^{\ast}_{\prec_{\mathcal{A}}}, L^{\ast}_{\succ_{\mathcal{A}}}, 
R^{\ast}_{\prec_{\mathcal{A}^{\ast}}},L^{\ast}_{\succ_{\mathcal{A}^{\ast}}})$ is a matched pair of the antiassociative algebras, 
it is straightforward to show that the bilinear form given by Eq.(\ref{sympletic form}) is invariant on the antiassociative algebra 
$ \mathcal{A} \bowtie^{R^{\ast}_{\prec_{\mathcal{A}}}, L^{\ast}_{\succ_{\mathcal{A}}}}_{R^{\ast}_{\prec_{\mathcal{A}^{\ast}}}, 
L^{\ast}_{\succ_{\mathcal{A}^{\ast}}}} \mathcal{A}^{\ast} $ given by :
\beqs
(x + a^{\ast})\ast_{\mathcal{A} \oplus \mathcal{A}^{\ast}} (y + b^{\ast}) &=&(x \ast_{\mathcal{A}} y + R^{\ast}_{\prec_{\mathcal{A}^{\ast}}}(a^{\ast})y 
+ L^{\ast}_{\succ_{\mathcal{A}^{\ast}}}(b^{\ast})x)\cr 
&&+ (a^{\ast} \ast_{\mathcal{A}^{\ast}} b^{\ast} + R^{\ast}_{\prec_{\mathcal{A}}}(x)b^{\ast} + L^{\ast}_{\succ_{\mathcal{A}}}(y)a^{\ast}).
\eeqs
In fact, we have 
\beqs
&&\omega[(x + a^{\ast})\ast_{\mathcal{A} \oplus \mathcal{A}^{\ast}} (y + b^{\ast}), z + c^{\ast}] 
+ \omega[(y + b^{\ast})\ast_{\mathcal{A} \oplus \mathcal{A}^{\ast}} (z + c^{\ast}), x + a^{\ast}] \cr
&&+\omega[(z + c^{\ast})\ast_{\mathcal{A} \oplus \mathcal{A}^{\ast}} (x + a^{\ast}), y + b^{\ast}] \cr
&=& -\langle x \ast_{\mathcal{A}} y + R^{\ast}_{\prec_{\mathcal{A}^{\ast}}}(a^{\ast})y + 
L^{\ast}_{\succ_{\mathcal{A}^{\ast}}}(b^{\ast})x, c^{\ast} \rangle + \langle a^{\ast} 
\ast_{\mathcal{A}^{\ast}} b^{\ast} + R^{\ast}_{\prec_{\mathcal{A}}}(x)b^{\ast}
 + L^{\ast}_{\succ_{\mathcal{A}}}(y)a^{\ast}, z \rangle \cr
&& -\langle y \ast_{\mathcal{A}} z + R^{\ast}_{\prec_{\mathcal{A}^{\ast}}}(b^{\ast})z + 
L^{\ast}_{\succ_{\mathcal{A}^{\ast}}}(c^{\ast})y, a^{\ast} \rangle + \langle b^{\ast} 
\ast_{\mathcal{A}^{\ast}} c^{\ast} + R^{\ast}_{\prec_{\mathcal{A}}}(y)c^{\ast} 
+ L^{\ast}_{\succ_{\mathcal{A}}}(z)b^{\ast}, x \rangle \cr
&&-\langle z \ast_{\mathcal{A}} x + R^{\ast}_{\prec_{\mathcal{A}^{\ast}}}(c^{\ast})x + 
L^{\ast}_{\succ_{\mathcal{A}^{\ast}}}(a^{\ast})z, b^{\ast} \rangle 
+ \langle c^{\ast} \ast_{\mathcal{A}^{\ast}} a^{\ast} + R^{\ast}_{\prec_{\mathcal{A}}}(z)a^{\ast} + 
L^{\ast}_{\prec_{\mathcal{A}}}(x)c^{\ast},  y \rangle \cr
&=& -\langle x \prec_{\mathcal{A}} y, c^{\ast} \rangle  -\langle x \succ_{\mathcal{A}} y,
 c^{\ast} \rangle -\langle c^{\ast} \prec_{\mathcal{A}^{\ast}} a^{\ast}, y \rangle - 
\langle b^{\ast} \succ_{\mathcal{A}^{\ast}} c^{\ast}, x  \rangle + \langle a^{\ast} \prec_{\mathcal{A}^{\ast}} b^{\ast}, 
 z  \rangle \cr
&&+ \langle a^{\ast} \succ_{\mathcal{A}^{\ast}} b^{\ast}, z  \rangle + \langle z \prec_{\mathcal{A}} x, 
b^{\ast} \rangle + \langle y \succ_{\mathcal{A}} z, a^{\ast} \rangle
 - \langle y \succ_{\mathcal{A}} z, a^{\ast} \rangle - \langle y \prec_{\mathcal{A}} z, a^{\ast} \rangle\cr
&&   - \langle a^{\ast} \prec_{\mathcal{A}^{\ast}} b^{\ast}, z \rangle - \langle a^{\ast}
\succ_{\mathcal{A}^{\ast}} c^{\ast}, y \rangle + \langle b^{\ast} 
\prec_{\mathcal{A}^{\ast}} c^{\ast}, x \rangle + \langle b^{\ast} 
\succ_{\mathcal{A}^{\ast}} c^{\ast}, x \rangle \cr
&& + \langle x \prec_{\mathcal{A}} y, c^{\ast} \rangle  
+ \langle z \succ_{\mathcal{A}} x, b^{\ast} \rangle - \langle z \prec_{\mathcal{A}} x, b^{\ast} \rangle
 - \langle z \succ_{\mathcal{A}} x, b^{\ast} \rangle   \cr
&& -\langle b^{\ast} \prec_{\mathcal{A}^{\ast}} c^{\ast},
  x \rangle -  \langle a^{\ast}\succ_{\mathcal{A}^{\ast}} b^{\ast}, z \rangle  +  \langle c^{\ast} \prec_{\mathcal{A}^{\ast}} a^{\ast}, 
 y \rangle  \cr
&& +  \langle a^{\ast} \succ_{\mathcal{A}^{\ast}} c^{\ast}, y \rangle  + \langle y \prec_{\mathcal{A}} z,
 a^{\ast} \rangle  + \langle x \succ_{\mathcal{A}} y, c^{\ast} \rangle \cr
&=& 0.
\eeqs
Conversely, if there exists a double construction of the sympletic antiassociative algebra associated to $(\mathcal{A}, \ast_{\mathcal{A}})$ 
and $(\mathcal{A}, \ast_{\mathcal{A}^{\ast}})$, then  $(\mathcal{A}, \mathcal{A}^{\ast}, R^{\ast}_{\prec_{\mathcal{A}}}, 
L^{\ast}_{\succ_{\mathcal{A}}}, R^{\ast}_{\prec_{\mathcal{A}^{\ast}}},L^{\ast}_{\succ_{\mathcal{A}^{\ast}}})$ is a matched pair of the
 antiassociative algebras  given by the following equations:
\beq \label{eq1.28}
R^{\ast}_{\prec_{\mathcal{A}}}(x)(a^{\ast} \ast_{\mathcal{A}^{\ast}} b^{\ast}) 
= -R^{\ast}_{\prec_{\mathcal{A}}}(L_{\prec_{\mathcal{A}^{\ast}}}(a^{\ast})x)b^{\ast} -
(R^{\ast}_{\prec_{\mathcal{A}}}(x)a^{\ast}) \ast_{\mathcal{A}^{\ast}} b^{\ast},
\eeq
\beq \label{eq1.29}
L^{\ast}_{\succ_{\mathcal{A}}}(x)(a^{\ast} \ast_{\mathcal{A}^{\ast}} b^{\ast}) =  
-L^{\ast}_{\succ_{\mathcal{A}}}(R_{\prec_{\mathcal{A}}}(b^{\ast})x)a^{\ast} - 
a^{\ast} \ast_{\mathcal{A}^{\ast}} (L^{\ast}_{\succ_{\mathcal{A}}}(x)b^{\ast}), 
\eeq
\beq \label{eq1.30}
 R^{\ast}_{\prec_{\mathcal{A}^{\ast}}}(a^{\ast})(x \ast_{\mathcal{A}} y^{\ast}) =  -
R^{\ast}_{\prec_{\mathcal{A}^{\ast}}}(L_{\succ_{\mathcal{A}}}(x)a^{\ast})y -
 (R^{\ast}_{\prec_{\mathcal{A}^{\ast}}}(a^{\ast})x) \ast_{\mathcal{A}} y, 
\eeq
\beq \label{eq1.31}
 L^{\ast}_{\succ_{\mathcal{A}^{\ast}}}(a^{\ast})(x \ast_{\mathcal{A}} y^{\ast}) =  
-L^{\ast}_{\succ_{\mathcal{A}^{\ast}}}(R_{\prec_{\mathcal{A}}}(y)a^{\ast})x -
 x \ast_{\mathcal{A}} (L^{\ast}_{\succ_{\mathcal{A}^{\ast}}}(a^{\ast})y),  
\eeq
\beq \label{eq1.32}
&& R^{\ast}_{\prec_{\mathcal{A}}}(R^{\ast}_{\prec_{\mathcal{A}^{\ast}}}(a^{\ast})x)b^{\ast} + 
(L^{\ast}_{\prec_{\mathcal{A}}}(x)a^{\ast})\ast_{\mathcal{A}^{\ast}}b^{\ast} \cr
&& + L^{\ast}_{\succ_{A}} 
(L^{\ast}_{\succ_{\mathcal{A}^{\ast}}}(b^{\ast})x)a^{\ast} + a^{\ast} \ast_{\mathcal{A}^{\ast}}(R^{\ast}_{\prec_{\mathcal{A}}}(x)b^{\ast}) = 0, 
\eeq
\beq \label{eq1.33}
&& R^{\ast}_{\prec_{\mathcal{A}}}(R^{\ast}_{\prec_{\mathcal{A}^{\ast}}}(x)a^{\ast})y + 
(L^{\ast}_{\succ_{\mathcal{A}^{\ast}}}(a^{\ast})x)\ast_{A}y \cr
&& + L^{\ast}_{\succ_{A^{\ast}}}  (L^{\ast}_{\succ_{A}}(y)a^{\ast})x + x \ast_{A}(R^{\ast}_{\prec_{A^{\ast}}}(a^{\ast})y) = 0
\eeq
since the operation $ \ast_{A \oplus A^{\ast}} $ is antiassociative.
$ \hfill \square $
\begin{corollary}
Let $ (\mathcal{A}, \succ, \prec) $ be an antidendriform algebra and $ (R^{\ast}_{\prec}, L^{\ast}_{\succ}) $ 
the bimodule of the associated antiassociative algebra $ (\mathcal{A}, \ast) $. Then, 
 $ (T(\mathcal{A}) = \mathcal{A} \ltimes _{R^{\ast}_{\prec}, L^{\ast}_{\succ}} \mathcal{A}^{\ast}, \omega )
  $ is a double construction of the sympletic antiassociative algebra. Conversely, let $ (T(\mathcal{A}) = \mathcal{A} \bowtie \mathcal{A}^{\ast}, \omega) $ 
be a double construction of the sympletic antiassociative algebra. If $ \mathcal{A}^{\ast} $ is 
  an ideal of $ T(\mathcal{A}) $, then $ \mathcal{A}^{\ast} $ is a trivial antiassociative algebra and hence $ T(\mathcal{A}) $ is isomorphic 
to the semidirect product $ \mathcal{A} \ltimes_{L_{T(\mathcal{A})},R_{T(\mathcal{A})}} \mathcal{A}^{\ast} $. Furthermore, this double
 construction of the sympletic antiassociative algebra is isomorphic to the double construction of the sympletic antiassociative algebra $ (T(\mathcal{A}) = 
\mathcal{A} \ltimes _{R^{\ast}_{\prec}, L^{\ast}_{\succ}}, \omega ) $ and the antidendriform algebra structure on
   $  \mathcal{A} $ is given by $ \omega $ from the equation (\ref{eq53}).
\end{corollary}
\textbf{Proof:} $(\mathcal{A}, \mathcal{A}^{\ast}, R^{\ast}_{\prec}, L^{\ast}_{\succ}, 0, 0)$,  with
 the antiassociative algebra structure on $ \mathcal{A}^{\ast} $ being trivial,  is always a matched pair of
  antiassociative algebras and, the first half of this Proof follows immediately.
 Conversely, if $ \mathcal{A}^{\ast} $ is an ideal, then, for any $ a^{\ast}, b^{\ast} \in 
 \mathcal{A}^{\ast} $, it follows that if $ T(\mathcal{A}) \ast a^{\ast}, b^{\ast} \ast T(\mathcal{A}) \in \mathcal{A}^{\ast} $, 
  then $ \omega(a^{\ast} \ast b^{\ast}, T(\mathcal{A})) = - \omega(T(\mathcal{A})\ast a^{\ast}, b^{\ast})
   - \omega( b^{\ast} \ast T(\mathcal{A}), a^{\ast}) = 0 $. Thus, 
 $ a^{\ast} \ast b^{\ast} = 0 $. Hence,  $ T(\mathcal{A}) $ is isomorphic to $ \mathcal{A} \ltimes_{L_{T(\mathcal{A})},R_{T(\mathcal{A})}} 
   \mathcal{A}^{\ast} $.
   It follows that $(T(\mathcal{A}) = \mathcal{A} \bowtie \mathcal{A}^{\ast},
    \omega)$ is isomorphic to the double construction of the sympletic antiassociative algebra $(T(\mathcal{A}) = 
\mathcal{A} \ltimes_{R^{\ast}_{\prec}, L^{\ast}_{\succ}} \mathcal{A}^{\ast}, \omega)$.

$ \hfill \square $
\begin{theorem}\label{theo4.1.7}
Let $(\mathcal{A}, \succ_{\mathcal{A}}, \prec_{\mathcal{A}})$ be an antidendriform algebra and 
$(\mathcal{A}, \ast_{\mathcal{A}}) $ the associated antiassociative algebra. Suppose that there is an antidendriform algebra structure
 $  "\succ_{\mathcal{A}^{\ast}}, \prec_{\mathcal{A}^{\ast}} " $ on its dual space $ \mathcal{A}^{\ast} $ and 
$ (\mathcal{A}^{\ast}, \ast_{\mathcal{A}^{\ast}}) $ is the associated antiassociative algebra. Then, 
  \beqs(\mathcal{A}, \mathcal{A}^{\ast}, R^{\ast}_{\prec_{\mathcal{A}}},L^{\ast}_{\succ_{\mathcal{A}}}, 
R^{\ast}_{\prec{\mathcal{A}^{\ast}}}, L^{\ast}_{\succ_{\mathcal{A}^{\ast}}})\eeqs
is a matched pair of antiassociative algebras if and only if 
\beqs
&&(\mathcal{A}, \mathcal{A}^{\ast}, R^{\ast}_{\succ_{\mathcal{A}}} + R^{\ast}_{\prec_{\mathcal{A}}}, -L^{\ast}_{\prec_{\mathcal{A}}}, -R^{\ast}_{\succ_{\mathcal{A}}}, L^{\ast}_{\succ_{\mathcal{A}}} + L^{\ast}_{\prec_{\mathcal{A}}},  R^{\ast}_{\succ_{\mathcal{A}^{\ast}}} 
+ R^{\ast}_{\prec_{\mathcal{A}^{\ast}}}, \cr
&& - L^{\ast}_{\prec_{\mathcal{A}^{\ast}}}, -R^{\ast}_{\succ_{\mathcal{A}^{\ast}}}, L^{\ast}_{\succ_{\mathcal{A}^{\ast}}} + L^{\ast}_{\prec_{\mathcal{A}^{\ast}}})
\eeqs  is a matched pair of antidendriform algebras.
\end{theorem}
\textbf{Proof:  } 
The   necessary condition  follows from the Corollary \ref{corollary 3.2.7}. We need to prove the  
 sufficient condition  only. 
If $(\mathcal{A}, \mathcal{A}^{\ast}, R^{\ast}_{\prec_{\mathcal{A}}}, L^{\ast}_{\succ_{\mathcal{A}}}, R^{\ast}_{\prec{\mathcal{A}^{\ast}}}, 
L^{\ast}_{\succ_{\mathcal{A}^{\ast}}})$ is a matched pair of antiassociative algebras, then $ ( \mathcal{A} \bowtie^{R^{\ast}_{\prec_{\mathcal{A}}}, 
L^{\ast}_{\succ_{\mathcal{A}}}}_{R^{\ast}_{\prec_{\mathcal{A}^{\ast}}}, L^{\ast}_{\succ_{\mathcal{A}^{\ast}}}} \mathcal{A}^{\ast} , \omega) $ 
is a double construction of the sympletic antiassociative algebra. Hence,  there exists a compatible antidendriform algebra structure on 
$  ( \mathcal{A} \bowtie^{R^{\ast}_{\prec_{\mathcal{A}}}, L^{\ast}_{\succ_{\mathcal{A}}}}_{R^{\ast}_{\prec_{\mathcal{A}^{\ast}}}, 
L^{\ast}_{\succ_{\mathcal{A}^{\ast}}}} \mathcal{A}^{\ast} )$ given by (\ref{eq53}). By a simple and direct computation, we show that 
$ \mathcal{A} $ and $ \mathcal{A}^{\ast} $ are its subalgebras and the other products are given by
\beqs
x \succ a^{\ast} &=& ( R^{\ast}_{\succ_{\mathcal{A}}} + R^{\ast}_{\prec_{\mathcal{A}}})(x)a^{\ast} - L^{\ast}_{\prec_{\mathcal{A}^{\ast}}}(a^{\ast})x, \cr \cr
x \prec a^{\ast} &=& - R^{\ast}_{\succ_{\mathcal{A}}} (x)a^{\ast} + (L^{\ast}_{\succ_{\mathcal{A}^{\ast}}} + 
L^{\ast}_{\prec_{\mathcal{A}^{\ast}}})(a^{\ast})x, \cr  \cr
a^{\ast} \succ x &=& ( R^{\ast}_{\succ_{\mathcal{A}^{\ast}}} + R^{\ast}_{\prec_{\mathcal{A}^{\ast}}})(a^{\ast})x - 
L^{\ast}_{\prec_{\mathcal{A}}}(x)a^{\ast}, \cr  \cr
a^{\ast} \prec x &=& -  R^{\ast}_{\succ_{\mathcal{A}^{\ast}}}(a^{\ast})x + (L^{\ast}_{\succ_{\mathcal{A}}} + 
L^{\ast}_{\prec_{\mathcal{A}}})(x)a^{\ast},
\eeqs
for any $ x \in \mathcal{A}, a^{\ast} \in \mathcal{A}^{\ast} $. Therefore, 
\beqs
&&(\mathcal{A}, \mathcal{A}^{\ast}, R^{\ast}_{\succ_{\mathcal{A}}} + R^{\ast}_{\prec_{\mathcal{A}}}, -L^{\ast}_{\prec_{\mathcal{A}}}, -
R^{\ast}_{\succ_{\mathcal{A}}}, L^{\ast}_{\succ_{\mathcal{A}}} + L^{\ast}_{\prec_{\mathcal{A}}},  R^{\ast}_{\succ_{\mathcal{A}^{\ast}}} + 
R^{\ast}_{\prec_{\mathcal{A}^{\ast}}}, \cr
&& - L^{\ast}_{\prec_{\mathcal{A}^{\ast}}}, -R^{\ast}_{\succ_{\mathcal{A}^{\ast}}},L^{\ast}_{\succ_{\mathcal{A}^{\ast}}} + 
L^{\ast}_{\prec_{\mathcal{A}^{\ast}}})
\eeqs  
is a matched pair of antidendriform algebras.
\begin{theorem}\label{theo4.2.4}
Let $(\mathcal{A}, \prec_{\mathcal{A}}, \succ_{\mathcal{A}})$ and $(\mathcal{A}^{\ast}, \prec_{\mathcal{A}^{\ast}}, \succ_{\mathcal{A}^{\ast}} )$ 
be two antidendriform algebras. Let $(\mathcal{A}, \ast_{\mathcal{A}})$ and   $(\mathcal{A}^{\ast}, \ast_{\mathcal{A}^{\ast}})$ be the 
  corresponding
associated
 antiassociative algebras. Then, the following conditions
\benum
\item[(1)] There is a double construction of the sympletic  antiassociative algebras associated to $(\mathcal{A}, \ast_{\mathcal{A}})$ and 
$(\mathcal{A}^{\ast}, \ast_{\mathcal{A}^{\ast}});$ 
\item[(2)]$ (\mathcal{A}, \mathcal{A}^{\ast}, R^{\ast}_{\prec_{\mathcal{A}}},  L^{\ast}_{\succ_{\mathcal{A}}},  R^{\ast}_{\prec_{\mathcal{A}^{\ast}}}, 
 L^{\ast}_{\succ_{\mathcal{A}^{\ast}}}) $ is a matched pair of the antiassociative algebras;
\item[(3)] $ (\mathcal{A}, \mathcal{A}^{\ast}, R^{\ast}_{\succ_{\mathcal{A}}} + R^{\ast}_{\prec_{\mathcal{A}}}, 
- L^{\ast}_{\prec_{\mathcal{A}}}, -R^{\ast}_{\succ_{\mathcal{A}}}, L^{\ast}_{\succ_{\mathcal{A}}} +
 L^{\ast}_{\prec_{\mathcal{A}}}, R^{\ast}_{\succ_{\mathcal{A}^{\ast}}} + R^{\ast}_{\prec_{\mathcal{A}^{\ast}}}, 
-L^{\ast}_{\prec_{\mathcal{A}^{\ast}}}, -R^{\ast}_{\succ_{\mathcal{A}^{\ast}}}, L^{\ast}_{\succ_{\mathcal{A}^{\ast}}}
 + L^{\ast}_{\prec_{\mathcal{A}^{\ast}}}) $ is a matched pair of antidendriform algebras;
\eenum
 are equivalent.
\end{theorem} 
\textbf{Proof:  }
This follows from Theorems \ref{theo4.1.5} and \ref{theo4.1.7}.

$ \hfill \square $
\section{Classification of 2-dimensional antiassociative algebras and double constructions}
In this section, we investigate the classification of 2-dimensional antiassociative algebras and some quadratic and sympletic double contructions. Classification of antiassociative algegras has  been done firstly in \cite{[MAR]}, but unfortunately with some errors. In fact, the authors found only two classes; comparatively we found four classes and, the proof is given in the following part.
Let $\mathcal{A} $ be an antiassociative algebra such that there is an antiassociative structure $"\circ"$ on its dual space $\mathcal{A}^{\ast}$ spanned by $\{e_1,e_2\}$ and $\{e_1^{\ast},e_2^{\ast}\}$ respectively. Formula (\ref{ant}) leads to the following relations:\\
%\begin{align*}
%&(e_1\cdot e_1)\cdot e_1=-e_1\cdot (e_1\cdot e_1), & (e_1\cdot e_1)\cdot e_2=-e_1\cdot (e_1\cdot e_2),\\
%&(e_1\cdot e_2)\cdot e_1=-e_1\cdot (e_2\cdot e_1), & (e_1\cdot e_2)\cdot e_2=-e_1\cdot (e_2\cdot e_2),\\
%&(e_2\cdot e_1)\cdot e_1=-e_2\cdot (e_1\cdot e_1), & (e_2\cdot e_1)\cdot e_2=-e_2\cdot (e_1\cdot e_2),\\
%&(e_2\cdot e_2)\cdot e_1=-e_2\cdot (e_2\cdot e_1), & (e_2\cdot e_2)\cdot e_2=-e_2\cdot (e_2\cdot e_2).
%\end{align*}
\beq
(e_i\cdot e_j)\cdot e_k=e_i\cdot (e_j\cdot e_k),\quad\quad where\quad i,j,k=1,2.
\eeq
Let $e_1\cdot e_1=a_1e_1+a_2e_2$, $e_1\cdot e_2=b_1e_1+b_2e_2$, $e_2\cdot e_1=c_1e_1+c_2e_2$, $e_2\cdot e_2=d_1e_1+d_2e_2$ where $a_1, a_2, b_1, b_2, c_1, c_2 d_1, d_2 \in \C$.

\begin{proposition}
There are four non-isomorphic 2-dimensional antiassociative
algebras $\mathcal{A} $ given by the following:
\beq\label{class}
e_i\cdot e_j=0,\quad\quad e_1\cdot e_1=e_2,\quad\quad e_2\cdot e_1=e_2,\quad\quad e_2\cdot e_2=e_1.
\eeq
\end{proposition}
\begin{proof}
Let $\A$ be a 2-dimensional antiassociative
algebra  with basis $\{e_1,e_2\}$. Suppose $x,y,z\in \A $ such that $x=x_1e_1+x_2e_2$, $y=y_1e_1+y_2e_2$ and $z=z_1e_1+z_2e_2$ with $x_1,x_2,y_1,y_2,z_1,z_2\in \C$. By antiassociativity and ignoring the coefficients, we get the relations $$(e_i\cdot e_j)\cdot e_k=-e_i\cdot(e_j\cdot e_k),\quad i,j,k=1,2.$$
Setting $e_1e_1 =a_1e_1+a_2e_2$, $e_1e_2=b_1e_1+b_2e_2$, $e_2e_1=c_1e_1+c_2e_2$ and $e_2e_2=d_1e_1+d_2e_2$, in the previous relations one have  the following eight relations equivalent to a system of 32 equations:
$\begin{cases}
~(a_1^2 +a_2c_1)e_1+(a_1a_2+a_2c_2)e_2=-(a_1^2+a_2b_1)e_1-(a_1a_2+a_2b_2)e_2,\\
~(a_1b_1+a_2d_1)e_1+(a_1b_2+a_2d_2)e_2=-(b_1a_1+b_2b_1)e_1-(b_1a_2+b_2^2)e_2,\\
~(b_1a_1+b_2c_1)e_1+(b_1a_2+b_2c_2)e_2=-(c_1a_1+c_2b_1)e_1-(c_1a_2+c_2b_2)e_2,\\
(b_1^2+b_2d_1)e_1+(b_1b_2+b_2d_2)e_2=-(d_1a_1+d_2b_1)e_1-(d_1a_2+d_2b_2)e_2,\\
~(b_1a_1+b_2c_1)e_1+(b_1a_2+b_2c_2)e_2=-(c_1a_1+c_2b_1)e_1-(c_1a_2+c_2b_2)e_2,\\
~(c_1b_1+c_2d_1)e_1+(c_1b_2+c_2d_2)e_2=-(b_1c_1+b_2d_1)e_1-(b_1c_2+b_2d_2)e_2,\\
~(d_1a_1+d_2c_1)e_1+(d_1a_2+d_2c_2)e_2=-(c_1^2+c_2d_1)e_1-(d_1c_2+d_2^2)e_2
\end{cases}$.\\

For 
\begin{center}
$a_1=0$\\
$\quad\quad\quad\quad\quad\quad\quad a_2=0\Rightarrow b_2=0$\\
$\quad\quad\quad\quad\quad b_1=0$\\
$\quad\quad\quad\quad\quad\quad\quad c_1=0$\\
$\quad\quad\quad\quad\quad\quad\quad \quad c_2=0$\\
$\quad\quad\quad\quad\quad\quad\quad\quad\quad\quad\quad\quad \quad \quad d_1\neq 0\Rightarrow d_2=0$\\
$\quad\quad\quad\quad \quad \quad class:\quad e_2e_2=d_1e_1$.
\end{center}
For 
\begin{center}
$a_1=0$\\
$\quad\quad\quad\quad\quad\quad\quad a_2=0\Rightarrow b_2=0$\\
$\quad\quad\quad\quad\quad b_1=0$\\
$\quad\quad\quad\quad\quad\quad\quad c_1=0$\\
$\quad\quad\quad\quad\quad\quad\quad \quad c_2=0$\\
$\quad\quad\quad\quad \quad \quad class:\quad e_2e_1=c_2e_2$.
\end{center}
For
\begin{center}
$a_1=0$\\
$\quad\quad a_2=0$\\
 $\quad\quad\quad\quad\quad\quad\quad\quad\quad\quad\quad\quad\quad\quad\quad\quad\quad b_1=0\Rightarrow c_1=c_2=b_2=d_1=d_2=0$\\
$\quad\quad\quad\quad \quad \quad class:\quad e_1e_1=a_1e_2$.
\end{center}
The last class is the trivial one ie. $e_ie_j=0$.
The other cases lead to an absurdity. Therefore, by isomorphism we get the following four classes: $e_ie_j=0$, $e_1e_1=e_2$, $e_2e_1=e_2$ and $e_2e_2=e_1$.
\end{proof}
Now, we discuss the quadratic antiassociative algebra structure on the direct sum $A\oplus A^{\ast}$.\\
$Case(I).$ $e_1\cdot e_1=e_2$. 
The product on the dual space is given by:
$
 e_2^{\ast}\circ e_1^{\ast}=e_2^{\ast}.
$
Using relation (\ref{product of matched pair}) when  $ l_{\mathcal{A}} = R^{\ast}, r_{\mathcal{A}} = L^{\ast}, l_{\mathcal{B}} = l_{\mathcal{A}^{\ast}} = R^{\ast}_{\circ}, r_{\mathcal{B}} = r_{\mathcal{A}^{\ast}} = L^{\ast}_{\circ} $, we obtain the double construction of quadratic antiassociative algebra $ ( \mathcal{A}\oplus \mathcal{A}^{\ast}, \ast, B) $ associated to $ (\mathcal{A}, \cdot) $ and $ (\mathcal{A}^{\ast}, \circ)$ given explicitly by the following relations:
\begin{align*}
(e_1+e_1^{\ast})\ast (e_1+e_1^{\ast})&=(e_1\cdot e_1 +R^{\ast}_{\circ}(e_1^{\ast})e_1+L^{\ast}_{\circ}(e_1^{\ast})e_1)+(e_1^{\ast}\circ e_1^{\ast}+R_{\cdot}^{\ast}(e_1)e_1^{\ast}+L_{\cdot}^{\ast}(e_1)e_1^{\ast} ),\\
&=e_2,\\
(e_1+e_1^{\ast})\ast (e_1+e_2^{\ast})&=(e_1\cdot e_1 +R^{\ast}_{\circ}(e_1^{\ast})e_1+L^{\ast}_{\circ}(e_2^{\ast})e_1)+(e_1^{\ast}\circ e_2^{\ast}+R_{\cdot}^{\ast}(e_1)e_2^{\ast}+L_{\cdot}^{\ast}(e_1)e_1^{\ast} ),\\
&=e_2 + e_1^{\ast},\\
(e_1+e_1^{\ast})\ast (e_2+e_1^{\ast})&=(e_1\cdot e_2 +R^{\ast}_{\circ}(e_1^{\ast})e_2+L^{\ast}_{\circ}(e_1^{\ast})e_1)+(e_1^{\ast}\circ e_1^{\ast}+R_{\cdot}^{\ast}(e_1)e_1^{\ast}+L_{\cdot}^{\ast}(e_2)e_1^{\ast} ),\\
&=e_2,\\
(e_1+e_1^{\ast})\ast (e_2+e_2^{\ast})&=(e_1\cdot e_2 +R^{\ast}_{\circ}(e_1^{\ast})e_2+L^{\ast}_{\circ}(e_2^{\ast})e_1)+(e_1^{\ast}\circ e_2^{\ast}+R_{\cdot}^{\ast}(e_1)e_2^{\ast}+L_{\cdot}^{\ast}(e_2)e_1^{\ast} ),\\
&=e_2+e_1^{\ast},\\
(e_1+e_2^{\ast})\ast (e_1+e_1^{\ast})&=(e_1\cdot e_1 +R^{\ast}_{\circ}(e_2^{\ast})e_1+L^{\ast}_{\circ}(e_1^{\ast})e_1)+(e_2^{\ast}\circ e_1^{\ast}+R_{\cdot}^{\ast}(e_1)e_1^{\ast}+L_{\cdot}^{\ast}(e_1)e_1^{\ast} ),\\
&=e_2+ e_2^{\ast},\\
(e_1+e_2^{\ast})\ast (e_1+e_2^{\ast})&=(e_1\cdot e_1 +R^{\ast}_{\circ}(e_2^{\ast})e_1+L^{\ast}_{\circ}(e_2^{\ast})e_1)+(e_2^{\ast}\circ e_2^{\ast}+R_{\cdot}^{\ast}(e_1)e_2^{\ast}+L_{\cdot}^{\ast}(e_1)e_2^{\ast} ),\\
&=e_2+2e_1^{\ast},\\
(e_2+e_1^{\ast})\ast (e_1+e_1^{\ast})&=(e_2\cdot e_1 +R^{\ast}_{\circ}(e_1^{\ast})e_1+L^{\ast}_{\circ}(e_2^{\ast})e_1)+(e_1^{\ast}\circ e_1^{\ast}+R_{\cdot}^{\ast}(e_2)e_1^{\ast}+L_{\cdot}^{\ast}(e_1)e_1^{\ast} ),\\
&=0,\\
(e_2+e_1^{\ast})\ast (e_1+e_2^{\ast})&=(e_2\cdot e_1 +R^{\ast}_{\circ}(e_1^{\ast})e_1+L^{\ast}_{\circ}(e_2^{\ast})e_2)+(e_1^{\ast}\circ e_2^{\ast}+R_{\cdot}^{\ast}(e_2)e_2^{\ast}+L_{\cdot}^{\ast}(e_1)e_1^{\ast} ),\\
&=e_1,\\
(e_2+e_2^{\ast})\ast (e_1+e_1^{\ast})&=(e_2\cdot e_1 +R^{\ast}_{\circ}(e_2^{\ast})e_1+L^{\ast}_{\circ}(e_1^{\ast})e_2)+(e_2^{\ast}\circ e_1^{\ast}+R_{\cdot}^{\ast}(e_2)e_1^{\ast}+L_{\cdot}^{\ast}(e_1)e_2^{\ast} ),\\
&=e_2^{\ast} + e_1^{\ast},\\
(e_2+e_2^{\ast})\ast (e_1+e_2^{\ast})&=(e_2\cdot e_1 +R^{\ast}_{\circ}(e_2^{\ast})e_1+L^{\ast}_{\circ}(e_2^{\ast})e_2)+(e_2^{\ast}\circ e_2^{\ast}+R_{\cdot}^{\ast}(e_2)e_2^{\ast}+L_{\cdot}^{\ast}(e_1)e_2^{\ast} ),\\
&=e_1+e_1^{\ast},\\
(e_2+e_2^{\ast})\ast (e_2+e_1^{\ast})&=(e_2\cdot e_2 +R^{\ast}_{\circ}(e_2^{\ast})e_2+L^{\ast}_{\circ}(e_1^{\ast})e_2)+(e_2^{\ast}\circ e_1^{\ast}+R_{\cdot}^{\ast}(e_2)e_1^{\ast}+L_{\cdot}^{\ast}(e_2)e_2^{\ast} ),\\
&=e_2^{\ast},\\
(e_2+e_2^{\ast})\ast (e_2+e_2^{\ast})&=(e_2\cdot e_2 +R^{\ast}_{\circ}(e_2^{\ast})e_2+L^{\ast}_{\circ}(e_2^{\ast})e_2)+(e_2^{\ast}\circ e_2^{\ast}+R_{\cdot}^{\ast}(e_2)e_2^{\ast}+L_{\cdot}^{\ast}(e_2)e_2^{\ast} ),\\
&=e_1,\\
(e_1+e_2^{\ast})\ast (e_2+e_1^{\ast})&=(e_1\cdot e_2 +R^{\ast}_{\circ}(e_2^{\ast})e_2+L^{\ast}_{\circ}(e_1^{\ast})e_1)+(e_2^{\ast}\circ e_1^{\ast}+R_{\cdot}^{\ast}(e_1)e_1^{\ast}+L_{\cdot}^{\ast}(e_2)e_2^{\ast} ),\\
&=e_2^{\ast},\\
(e_1+e_2^{\ast})\ast (e_2+e_2^{\ast})&=(e_1\cdot e_2 +R^{\ast}_{\circ}(e_2^{\ast})e_2+L^{\ast}_{\circ}(e_2^{\ast})e_1)+(e_2^{\ast}\circ e_2^{\ast}+R_{\cdot}^{\ast}(e_1)e_2^{\ast}+L_{\cdot}^{\ast}(e_2)e_2^{\ast} ),\\
&=e_1^{\ast},\\
(e_2+e_1^{\ast})\ast (e_2+e_1^{\ast})&=(e_2\cdot e_2 +R^{\ast}_{\circ}(e_1^{\ast})e_2+L^{\ast}_{\circ}(e_1^{\ast})e_2)+(e_1^{\ast}\circ e_1^{\ast}+R_{\cdot}^{\ast}(e_2)e_1^{\ast}+L_{\cdot}^{\ast}(e_2)e_1^{\ast} ),\\
&=e_2,\\
(e_2+e_1^{\ast})\ast (e_2+e_2^{\ast})&=(e_2\cdot e_2 +R^{\ast}_{\circ}(e_1^{\ast})e_2+L^{\ast}_{\circ}(e_2^{\ast})e_2)+(e_1^{\ast}\circ e_2^{\ast}+R_{\cdot}^{\ast}(e_2)e_2^{\ast}+L_{\cdot}^{\ast}(e_2)e_1^{\ast} ),\\
&= e_1 + e_2.
\end{align*}
%%%%%%%%%%%%%%%%%%%%%%%%%%%%%%ù

%%%%%%%%%%%%%%%%%%%%%%%%%%%%%%%
$Case(II)$ $e_2\cdot e_1=e_2$. 
The product on the dual space is given by:
$
e_1^{\ast}\circ e_1^{\ast}= e_2^{\ast}.
$
Similarly, the double construction of quadratic antiassociative algebra $ ( \mathcal{A}\oplus \mathcal{A}^{\ast}, \ast, B) $ associated to $ (\mathcal{A}, \cdot) $ and $ (\mathcal{A}^{\ast}, \circ)$ is given explicitly by the following relations:
\begin{align*}
(e_1+e_1^{\ast})\ast (e_1+e_1^{\ast})&=e_2^{\ast} ,\\
(e_1+e_1^{\ast})\ast (e_1+e_2^{\ast})&=e_2^{\ast},\\
(e_1+e_1^{\ast})\ast (e_2+e_1^{\ast})&=e_1+e_2^{\ast},\\
(e_1+e_1^{\ast})\ast (e_2+e_2^{\ast})&=e_1+e_2^{\ast},\\
(e_1+e_2^{\ast})\ast (e_1+e_1^{\ast})&=0,\\
(e_1+e_2^{\ast})\ast (e_1+e_2^{\ast})&=e_2^{\ast},\\
(e_2+e_1^{\ast})\ast (e_1+e_1^{\ast})&=e_2+e_2^{\ast},\\
(e_2+e_1^{\ast})\ast (e_1+e_2^{\ast})&=e_2,\\
(e_2+e_2^{\ast})\ast (e_1+e_1^{\ast})&=e_2+e_1,\\
(e_2+e_2^{\ast})\ast (e_1+e_2^{\ast})&=e_2,\\
(e_2+e_2^{\ast})\ast (e_2+e_1^{\ast})&=e_1+e_1^{\ast},\\
(e_2+e_2^{\ast})\ast (e_2+e_2^{\ast})&=e_1^{\ast},\\
(e_1+e_2^{\ast})\ast (e_2+e_1^{\ast})&=e_1^{\ast},\\
(e_1+e_2^{\ast})\ast (e_2+e_2^{\ast})&=e_1^{\ast}+e_2^{\ast},\\
(e_2+e_1^{\ast})\ast (e_2+e_1^{\ast})&=2e_1+e_2^{\ast},\\
(e_2+e_1^{\ast})\ast (e_2+e_2^{\ast})&=e_1.
\end{align*}

$Case(III)$ The following relations $e_1\prec e_1=\lambda e_2,\ e_1\succ e_1=(1-\lambda) e_2, \lambda\in K$ define an antidendriform algebra associated to the antiassociative algebra $e_1\cdot e_1=e_2$. 
A double construction of sympletic antiassociative algebra $ ( \mathcal{A}\oplus \mathcal{A}^{\ast}, \ast, \omega) $ is given explicitly by the following relations:
$
e_1\ast e_1=e_2 ,\ e_1\ast e_2^{\ast}=\lambda e_1^{\ast} ,\ e_2^{\ast}\ast e_1=(1-\lambda)e_1^{\ast}.$

$Case(IV)$ The following relations $e_2\prec e_2=-e_1,\ e_2\succ e_2=e_1,$ define an antidendriform algebra associated to the antiassociative algebra $e_i\cdot e_j=0$. 
A double construction of sympletic antiassociative algebra $ ( \mathcal{A}\oplus \mathcal{A}^{\ast}, \ast, \omega) $ is given explicitly by the following relations:
$
e_1^{\ast}\ast e_2 = e_2^{\ast} ,\ e_2\ast e_1^{\ast}=-e_2^{\ast}.$

%\section{Concluding remarks}
%In this work, we discussed the theory of $q$-generalized associative algebras particularly the bimodules and matched pairs of $q$-generalized associative algebras. We fully discussed The double constructions of quadratic antisymetric infinitesimal antiassociative bialgebras. Finally, we computed for each class the corresponding antisymetric infinitesimal  antiassociative bialgebra structures and, the associated double constructions of quadratic antiassociative algebras.

\end{document}